\documentclass[11pt,addpoints]{article}
\usepackage{graphicx}
\graphicspath{{Figures/}}
\usepackage{amsfonts}
\usepackage{bbm}
\usepackage{amsmath}
\usepackage[table]{xcolor}
\usepackage[margin=2cm,bottom=2.5cm]{geometry}
\usepackage{changepage}
\usepackage{lscape}
\usepackage{multirow}
\usepackage{subcaption}
\usepackage{caption}
\usepackage{enumerate}
\usepackage{algorithm}
\usepackage{algorithmic}
\usepackage{xspace}
\usepackage{booktabs}
\usepackage{amsthm}

\newtheorem{theorem}{Theorem}

\newtheorem{proposition}{Proposition}

\newtheorem{definition}{Definition}
\newtheorem{remark}{Remark}
\newtheorem{example}{Example}

\usepackage[round]{natbib}
\usepackage{setspace}
\usepackage{hyperref}
\usepackage[affil-it]{authblk}

\newcommand{\WMCIG}{WMCIG\xspace}
\newcommand{\BIIG}{BIIG\xspace}

\begin{document}
\author[1]{K\"ubra Tan{\i}nm{\i}\c{s}\thanks{kuebra.taninmis\_ersues@jku.at}}
\author[2]{Markus Sinnl\thanks{markus.sinnl@jku.at}}

\affil[1]{Institute of Production and Logistics Management, Johannes Kepler University Linz, Austria}

\affil[2]{Institute of Production and Logistics Management, Johannes Kepler University Linz, Linz, Austria \\
		JKU Business School, Johannes Kepler University Linz, Austria}

\title{A branch-and-cut algorithm for submodular interdiction games}

\date{}

\maketitle
\begin{abstract}

	Many relevant applications from diverse areas such as marketing, wildlife conservation or defending critical infrastructure can be modeled as interdiction games. In this work, we introduce interdiction games whose objective is a monotone and submodular set function. Given a ground set of items, the leader interdicts the usage of some of the items of the follower in order to minimize the objective value achievable by the follower, who seeks to maximize a submodular set function over the uninterdicted items subject to knapsack constraints. 
	
	We propose an exact branch-and-cut algorithm for these kind of interdiction games.
	The algorithm is based on interdiction cuts which allow to capture the followers objective function value for a given interdiction decision of the leader and exploit the submodularity of the objective function. We also present extensions and liftings of these cuts and discuss additional preprocessing procedures. 
	
	We test our solution framework on the weighted maximal covering interdiction game and the bipartite inference interdiction game.	
	 For both applications, the improved variants of our interdiction cut perform significantly better than its basic version. For the weighted maximal covering interdiction game for which a mixed-integer bilevel linear programming (MIBLP) formulation is available, we compare the results with those of a state-of-the-art MIBLP solver. While the MIBLP solver yields a minimum of 54\% optimality gap within one hour, our best branch-and-cut setting solves all but 4 of 108 instances to optimality with a maximum of 3\% gap among unsolved ones.

\end{abstract}



\section{Introduction and Problem Definition}
A bilevel optimization problem involves two decision makers with conflicting objectives. The first decision maker who is called the \textit{leader} integrates the response of the \textit{follower}, i.e., the second decision maker, into her decision making process. While the leader has complete knowledge of the objective and the constraints of the follower, once she makes a decision the follower has the full information of her decision and decides accordingly. In other words, they play a sequential game which is known as a Stackelberg game \citep{von1952theory}. While many real world problems involving competition and non-cooperation can be addressed as bilevel optimization models, even the simplest version of bilevel problems is known to be $\mathcal{NP}$-hard \citep{jeroslow1985polynomial,ben1990computational}. A recent survey on bilevel optimization is presented by \citet{dempe2020bilevel}.

In this paper, we address a special class of the bilevel optimization problems called \textit{Interdiction Games} (IG). This kind of problems are two-player zero-sum Stackelberg games and have received considerable attention in recent years. In an IG, the aim of the leader is to attain the maximum deterioration in the follower's optimal objective value by \textit{interdicting} her decisions. IGs have applications in diverse areas such as marketing \citep{denegre2011interdiction}, wildlife conservation \citep{mc2016preventing, sefair2017defender} or defending critical infrastructure \citep{brown2006defending}.
Most of the IGs that have been studied so far are related to network interdiction where certain components of a network such as its edges or nodes are interdicted by the leader so that the follower cannot use them to achieve its objective. \citet{smith2020survey} present a comprehensive survey on network interdiction models. Other popular IGs are the knapsack interdiction problem \citep{denegre2011interdiction}, or the facility interdiction problem and its variants \citep{church2004identifying, aksen2014bilevel}. A more detailed review of IGs and state-of-the-art solution approaches is provided in Section \ref{section:Literature}.

\subsection{Problem Definition}
\label{section:ProblemDefinition}
In this study, we consider the class of IGs with a \emph{submodular} and \emph{monotone} objective function. Given a finite ground set $N$ (of \emph{items}), a function $z:2^N\rightarrow \mathbb{R}$ is called submodular if 
$z(S\cup \{i\}) -z(S)\geq z(T\cup \{i\}) -z(T)$, for all $S\subseteq T \subseteq N$ and $i\in N \setminus T$ (alternative definitions by \cite{nemhauser1978analysis} are provided in Section \ref{section:preliminaries}). The function $z$ is also monotone (non-decreasing) if $z(S)\leq z(T)$ for all $S\subseteq T \subseteq N$. 

Many problems including the maximal covering problem \citep{church1974maximal, vohra1993probabilistic}, uncapacitated facility location problem \citep{nemhauser1981maximizing}, influence maximization problem under linear threshold and independent cascade models \citep{kempe2003maximizing}, bipartite inference problem \citep{sakaue2018accelerated,salvagnin2019some}, assortment optimization problem \citep{kunnumkal2019tractable}, maximum capture location problem \citep{ljubic2018outer} or minimum variance sensor placement problem \citep{krause2008robust}, have submodular objective functions. Rank functions and weighted rank functions of matroids are also submodular \citep{schrijver2003combinatorial}.

In particular, we address IGs whose follower seeks to maximize a submodular and non-decreasing set function subject to knapsack constraints, which is known to be a $\mathcal{NP}$-hard problem \citep{cornuejols1977location}. The leader of the game interdicts the usage of a set of items in $N$ in order to minimize the follower's optimal objective value. The problem addressed is formulated as
\begin{align}
&\min_{x\in X} \, \max \big\{z(S): S\subseteq {N\setminus N_x}, C(S)\leq Q, \big\}
\label{eq:P1}
\end{align}
where $N_x=\{i\in N: x_i=1\}$ is the set of items that are not available to the follower under the interdiction strategy $x$. The set $X=\{x\in \{0,1\}^{n}:Ax\leq b\}$ is the feasible region of the leader, where $A$ and $b$ are a real valued matrix and a vector of appropriate dimensions and $n$ denotes the number of leader variables. The follower is constrained by knapsack constraints $C(S)\leq Q$ where $C(S)=\{c_\ell(S)= \sum_{i\in S} c^\ell_{i}, \ell=1, \ldots, L \}$ with $c^\ell_{i}\geq 0$ $\forall i, \ell$, and $Q$ is a vector of appropriate dimension. 

We note that 
all the problems mentioned above fall under the structure of Problem \eqref{eq:P1} and interdiction versions of these problems can be solved with our solution approach. 

\subsection{Contribution and Outline}

The main contribution of this study is an exact method for solving IGs with a submodular and non-decreasing objective function as given in \eqref{eq:P1}. Using properties of submodular functions we introduce \textit{submodular interdiction cuts (SICs)}. They are based on the value of the contribution to the objective value due to adding an item to a given subset of $N$, which is called \textit{marginal gain}. Problem \eqref{eq:P1} is reformulated as a single level problem using our SICs and solved within a branch-and-cut scheme.
We also propose various ways to lift our SICs and test the effectiveness of the resulting solution algorithms on the weighted maximal covering interdiction game and the bipartite inference interdiction game.

The outline of the paper is as follows. In the remainder of this section, we give a discussion of previous and related work. In Section \ref{section:SubmodularCuts}, we first recall basic properties of submodular functions and then introduce our basic SICs and show how to obtain a single level reformulation of Problem \eqref{eq:P1} using these cuts. Finally, we also introduce the problems used in the computational study in this section.
%
 %
In Section \ref{section:improvements}, we propose improved, lifted and alternative versions of our SICs and also give illustrational examples of their occurrence in the weighted maximal covering interdiction game and the bipartite inference interdiction game. Section \ref{section:implementation} contains implementation details of our branch-and-cut solution framework, including separation procedures for our SICs.
In Section \ref{section:Computational}, we present the computational results of our approach on the problem families selected as test bed. For the weighted maximal covering interdiction game, for which a mixed-integer bilevel linear programming (MIBLP) formulation is possible, we compare our approach against a state-of-the-art MIBLP solver.
 We conclude the paper with possible future research directions in Section \ref{section:Conclusion}.

\subsection{Previous and Related Work}
\label{section:Literature}

In some cases IGs can be formulated as MIBLPs, in which case they are solvable via general purpose MIBLP solvers such as the ones proposed by \cite{xu2014exact, lozano2017value, fischetti2017new, tahernejad2020branch}. 
On the other hand, there exist also specialized methods either for a specific problem type or for more general IGs. 
In various studies, the IG addressed has a linear follower problem and is formulated as a single level optimization problem via linear programming duality. 
This is the case in the works of \cite{wollmer1964removing}, \cite{wood1993deterministic} and \cite{morton2007models} where the maximum flow interdiction problem is addressed with the aim of analyzing the sensitivity of a transportation network, reducing the flow of drugs on a network, and stopping nuclear smuggling, respectively. 
Similarly, the shortest path interdiction problem \citep{golden1978problem,  israeli2002shortest, bayrak2008shortest} and the node deletion problem \citep{shen2012exact} which aims to damage the connectivity of a network, can be solved via duality-based approaches. 
Although this approach has been frequently used in IG modeling, many real world problems give rise to mixed-integer lower-level problems. Among them, there are problems that can still be formulated as an MIP due to their special structure such as the $r$-median interdiction problem \citep{church2004identifying}. Using the closest assignment constraints, the follower decision can be integrated to the leader's problem. 
Some variants like the one with partial interdiction addressed in \cite{aksen2014bilevel} still require MIBLP formulations. 

The $r$-interdiction covering problem introduced in \cite{church2004identifying} involves finding the facilities to interdict to maximize the coverage reduction. 
It has applications in determining critical existing emergency facilities such as fire stations or emergency communication systems.  
Since it involves a single decision maker, the attacker, the problem is not exactly an IG and can be formulated as an MIP. 
The IG version of this problem with a defender locating facilities after interdiction fulfills the requirements of our framework and is one of the applications we consider in our computational study (see Section \ref{section:Applications}).  
Facility location interdiction problems have also been considered within a fortification setting called \emph{defender-attacker-defender model} where the defender seeks to minimize the damage due to interdiction (see, e.g., \cite{brown2006defending, scaparra2008bilevel, scaparra2008exact, aksen2010budget} for the $r$-interdiction median with fortification; \cite{dong2010model} and \cite{roboredo2019exact} for $r$-interdiction covering with fortification). \cite{cappanera2011optimal} study shortest path interdiction with fortification. \cite{lozano2017backward} propose a sampling based exact method for a more general class of three-level fortification problems.

Another widely studied IG is the knapsack interdiction problem. In one version of this problem the leader's decision affects the follower's budget. \cite{brotcorne2013one} propose a dynamic programming based method and a single level formulation for this version. 
In a more commonly studied version introduced by \cite{denegre2011interdiction} the leader interdicts the usage of some items by the followers, which could have an application in corporate marketing strategies. \cite{denegre2011interdiction} develops a branch-and-cut scheme and \cite{caprara2016bilevel} propose an iterative algorithm for this variant of the knapsack problem. \cite{della2020exact} compute effective lower bounds on the optimal objective and utilize them to design an exact algorithm. 

Another interdiction problem which recently got more attention in literature is the clique interdiction problem. The problem involves minimizing the size of the maximum clique in a network, by interdicting, i.e., removing, a subset of its edges \citep{tang2016class, furini2021branch} or vertices \citep{furini2019maximum}. 

Finally, there also exists work on stochastic and robust versions of interdiction. For example, in \citep{cormican1998stochastic}, a stochastic network interdiction problem is considered.
In \citep{borrero2021modeling}, an attacker affects the objective function of the defender in an uncertain way. Two exact methods are proposed to solve the robust optimization problem of the defender who wants to be prepared for the worst case scenario.

There have been several studies focusing on generic methods to solve IGs. \cite{tang2016class} propose iterative algorithms for IGs with a mixed-integer follower problem. These algorithms are finitely convergent when the leader variables are restricted to take binary values. \cite{taninmis2020improved} improve the algorithm of \cite{tang2016class} for the binary bilevel problem case, using a covering based reformulation of the problem instead of a duality based one. \citet{fischetti2019interdiction} address IGs that satisfy an assumption called \emph{downward monotonicity}. 
They introduce a branch-and-cut approach based on efficient use of \textit{interdiction cuts} which previously have been used within problem specific solution frameworks in several studies including \citet{israeli2002shortest}, \citet{wood2010bilevel} and \citet{caprara2016bilevel}.
The problems we address and design solution approaches for in this work form a more general class of IGs addressed in \citet{fischetti2019interdiction}. The reason is that we allow the objective function of the IG to be linear or non-linear as long as it is submodular and non-decreasing, while a linear objective function of discrete decision variables can equivalently be expressed as a submodular non-decreasing set function given that the objective coefficients are non-negative.

\section{Submodular Interdiction Cuts}
\label{section:SubmodularCuts}

\subsection{Preliminaries on Submodular Functions}
\label{section:preliminaries}
Given a submodular function $z$, let $\rho_i(S)=z(S\cup \{i\})-z(S)$ be the \emph{marginal gain} due to adding $i\in N$ to set $S\subseteq N$. The marginal gain $\rho(\cdot)$ is non-increasing by definition of a submodular function. The following proposition gives alternative definitions for submodular functions.
\begin{proposition}
	\citep{nemhauser1978analysis}. If $z$ is a submodular function, then
	\begin{equation}
	z(T)\leq z(S) +\sum_{i\in T\setminus S}\rho_i(S) - \sum_{i\in S\setminus T}\rho_i(S\cup T \setminus \{i\})  \hspace{1cm} S, T \subseteq N, \label{eq:Nemhauser1}
	\end{equation}
	\begin{equation}
	z(T)\leq z(S) +\sum_{i\in T\setminus S}\rho_i(S\cap T) - \sum_{i\in S\setminus T}\rho_i(S\setminus \{i\})  \hspace{1cm} S, T \subseteq N. \label{eq:Nemhauser2}
	\end{equation}
	\label{prop:submodular}
\end{proposition}

\begin{proposition}
	\citep{nemhauser1978analysis}.
	If $z$ is a submodular and non-decreasing function, then 
	\begin{equation}
	\rho_i(S) \geq \rho_i(T)\geq 0  \hspace{1cm} S\subseteq T \subseteq N, i\in N, \label{eq:Nemhauser3}
	\end{equation}
	and the last term in \eqref{eq:Nemhauser1} can be removed to obtain the simpler inequality:
	\begin{equation}
	z(T)\leq z(S) +\sum_{i\in T\setminus S}\rho_i(S) \hspace{1cm} S, T \subseteq N. \label{eq:Nemhauser4}
	\end{equation}  
	
	\label{prop:submodular2}
\end{proposition}

\subsection{Single Level Reformulation of IGs and Basic Submodular Interdiction Cuts}


Let $\Phi(x)$ be the value function of the follower problem of \eqref{eq:P1}, i.e., $\Phi(x)=\max \big\{z(S): S\subseteq {N\setminus N_x}, C(S)\leq Q  \big\}$. Our problem can be reformulated as 
\begin{align}
	 \min \,  &w& \label{eq:SingleL-1}\\
	& w \geq \Phi(x) \\
	& Ax\leq b \\
	& x\in \{0,1\}^{n}. \label{eq:SingleL-4}
\end{align}

Rewriting the value function for given $x$ as $\Phi(x)=\max \big\{z(S)-\sum_{i\in S} M_i x_i: S\in \mathcal{S} \big\}$, where $\mathcal{S}=\{S\subseteq {N}: C(S)\leq Q\}$ is the set of all feasible follower solutions, allows expressing the feasible region of the follower independent from the leader's decision by penalizing infeasible solutions where $ \exists i\in S, x_i=1$ with big-$M_i$. Then \eqref{eq:P1} can be restated as
\begin{align}
 \min\, & w\\
	& w \geq z(\hat{S})-\sum_{i\in \hat{S}} M_i x_i & \hat{S}\in \mathcal{S} \label{eq:InterdictionCut} \\
	& Ax\leq b \\
	& x\in \{0,1\}^{n}.
\end{align}

We note that above reformulation follows the same ideas as 
proposed for IGs with a linear follower objective function (see e.g.,\citep{israeli2002shortest, caprara2016bilevel, fischetti2019interdiction}). In the linear case the \emph{interdiction cuts} \eqref{eq:InterdictionCut} can be written as $w\geq d^T\hat{y} - \sum_{i\in N}M_i x_i \hat{y}_i$ where $y$ is the vector of binary follower variables, with $\hat y$ being a follower solution, and $d$ is the vector of follower objective coefficients. \cite{fischetti2019interdiction} prove the validity of the interdiction cut when $M_i=d_i$ under some assumptions. In the following, we present valid cuts for our problem in the form of \eqref{eq:InterdictionCut} using the submodularity of $z(S)$. As is the case with a linear objective function the values of the big-$M$ coefficients determine the strength of the formulation. 
We thus propose various liftings and variants of our cuts in Section \ref{section:improvements}. We solve the reformulation with a branch-and-cut scheme, where our various SICs are separated for integer and fractional leader $x^*$, implementation details are discussed in Section \ref{section:implementation}.


\begin{theorem}
Given a follower solution $\hat{S}\in \mathcal{S}$, the following \emph{basic SIC} is valid for \eqref{eq:SingleL-1}--\eqref{eq:SingleL-4}.
\begin{equation}
	w\geq z(\hat{S})-\sum_{i\in \hat{S}} \rho_{i}(\emptyset)x_i
	\label{eq:basicCut}
\end{equation}
\label{theo:basicCut}
\end{theorem}
\begin{proof}
For any feasible leader solution $x \in X$, define the follower solution $S^\prime=\hat{S} \setminus N_x$. Because $S^\prime \subseteq N \setminus N_x$, and $C(S^\prime)\leq C(\hat{S})$ due to non-negativity of $c_i^\ell$, $S^\prime$ is a feasible solution for $x$.
Due to $z(S)$ being submodular and non-decreasing, and using \eqref{eq:Nemhauser4}, we have 
\begin{equation}
	z(\hat{S})\leq z(S^\prime)+\sum_{i\in \hat{S}\setminus S^\prime }\rho_{i}(S^\prime)=z(S^\prime)+\sum_{i\in \hat{S}}\rho_{i}(S^\prime)x_i.
\end{equation} 
 Thus we have 
\begin{equation}
	w\geq \Phi(x) \geq z(S^\prime) \geq	z(\hat{S})-\sum_{i\in \hat{S}}\rho_{i}(S^\prime)x_i \geq z(\hat{S})-\sum_{i\in \hat{S}}\rho_{i}(\emptyset)x_i,
\end{equation}
which shows that the basic SIC for $\hat{S}$ is satisfied for any leader solution $x$. The last inequality follows from the fact that $\rho_{i}(S^\prime)\leq \rho_{i}(\emptyset)$, $i \in N$. 
\end{proof}

\subsection{Exemplary Submodular Interdiction Games}
\label{section:Applications}
The following two problems are interdiction variants of submodular optimization problems, which will be used in our computational study. The \emph{weighted maximal covering problem (MCP)} is a classical problem in location science (see, e.g., \cite{church1974maximal, LaporteEtal2015}), where the goal is to open $k$ facilities in order to maximize the number of customers covered by these open facilities. In the proposed interdiction variant, which we denote as \emph{weighted maximal coverage interdiction game (\WMCIG)} the leader can interdict the opening of some facilities. Similar to interdiction variants of other facility location problems (see e.g., Section \ref{section:Literature}) applications of the \WMCIG are in critical infrastructure protection and facility location under competition. A formal definition of the problem is given below.

\begin{definition}[Weighted maximal coverage interdiction game  (\WMCIG)]

	We are given a set of $m$ customers $J$ with profits $p_j$, $j\in J$, a set of 
	potential facility locations $N$ and for each facility $i \in N$ the set $J(i) \subseteq J$ of customers that a facility at location $i$ covers. Moreover, we are given two integers $k$ and $B$. The problem of the follower is finding a set of $B$ locations to open a facility to maximize the profit of covered customers, where the profit for a set $S\subseteq N$ of open facilities is defined as $z(S)=\sum_{j\in J(S)}p_j$ where $J(S)=\cup_{i\in S} J(i)$, i.e., the profit obtained from customers that are covered by at least one of the facilities in $S$.
	The goal of the leader is to interdict $k$ facility locations such that the profit of the follower is minimized.	
\end{definition}

We note that for the MCP a compact mixed-integer programming formulation is known, and thus for the \WMCIG a MIBLP formulation can be obtained and the problem can be solved using a MIBLP-solver. This formulation is discussed in Section \ref{section:Computational} where we also provide a computational comparison between our branch-and-cut and solving \WMCIG as a MIBLP with a state-of-the-art MIBLP-solver.

The second problem we consider is the interdiction variant of the bipartite inference problem (BIP). The BIP is studied in \citet{alon2012optimizing, sakaue2018accelerated}, and \citet{salvagnin2019some}, with an application to the allocation of marketing budget among media channels in the former. Its interdiction version could represent a competitive setting where an existing firm tries to undermine the marketing activities of a newcomer.
Contrary to the MCP, for the BIP only a submodular formulation is known.

\begin{definition}[Bipartite Inference Interdiction Game (\BIIG)]
Given a set of items $N$, a set of targets $M$, and a bipartite graph $G=(N,M,A)$, the objective of the follower in the \BIIG is to select a set of items $S\subseteq N$ that maximizes the total activation probabilities of all targets
\begin{equation}
z(S)=\sum_{j\in M} p_S(j), \notag
\end{equation}
where 
\begin{equation}
p_S(j)=1- \prod_{i\in S: (i,j)\in A} (1-p_i) \notag
\end{equation}
denotes the activation probability of target $j$ and $p_i$ is the activating probability of item $i \in N$, independent of the target \citep{sakaue2018accelerated, salvagnin2019some}. The follower has a budget of $B$ to select items. The objective of the leader is to minimize the total activation probability by interdicting a set of items in $N$ subject to a cardinality constraint, at most $k$ items can be interdicted.
\end{definition}

\section{Improvements, Liftings and Variants of the Basic Submodular Interdiction Cut\label{section:improvements}}

In this section, we first show how to obtain an improved version of \eqref{eq:basicCut} by exploiting the diminishing gains property of submodular functions which implies that as the set expands the marginal gains decrease.

\begin{theorem}
Given an arbitrary ordering $(i_1,i_2,...,i_T)$ of the items in follower solution $\hat{S}\in \mathcal{S}$, let $\hat{S}_{(1)}=\emptyset$ and $\hat{S}_{(t)}=\{i_1,...,i_{t-1}\}$ for $2\leq t\leq T$. The following \emph{improved SIC} is valid for \eqref{eq:SingleL-1}--\eqref{eq:SingleL-4} and it dominates \eqref{eq:basicCut}.

\begin{equation}
	w\geq z(\hat{S})-\sum_{t=1}^T \rho_{i_t}(\hat{S}_{(t)})x_{i_t}
	\label{eq:ImprovedCut}
\end{equation}
\label{theo:ImprovedCut}
\end{theorem}
\begin{proof}

For any $x\in X$, define $S^\prime=\hat{S}\setminus N_x$ which is a feasible set for interdiction decision $x$ (see the proof of Theorem \eqref{theo:basicCut}). Let $S^\prime_{(t)}=\hat{S}_{(t)}\setminus N_x$ denote the items in $\hat{S}_{(t)}$ that are not interdicted in $x$, i.e., its maximal feasible subset. Notice that $S^\prime_{(t)} \subseteq S^\prime_{(t+1)}$. Using the same ordering of the items in $\hat{S}$, the objective value of $S^\prime$ can be computed incrementally by using the definition of a marginal gain. Starting with an empty set, increasing the objective value at each step by the marginal gain of the next item with respect to the current set yields the objective value of the final set, as used below.
\begin{equation}
\begin{split}
z(S^\prime) &= \rho_{i_1}(S^\prime_{(1)})(1-x_{i_1})+\rho_{i_2}(S^\prime_{(2)})(1-x_{i_2}) + \dots + \rho_{i_T}(S^\prime_{(T)})(1-x_{i_T})  \\
&= \sum_{t=1}^T \rho_{i_t}(S^\prime_{(t)})(1-x_{i_t}) \geq \sum_{t=1}^T \rho_{i_t}(\hat{S}_{(t)})(1-x_{i_t})
\end{split}
\label{eq:ImprovedCut3}
\end{equation}
Here, if an item $i_t$ is interdicted, its contribution is not included in the sum due to the $(1-x_{i_t})$ multiplier and $S^\prime_{(t)}=S^\prime_{(t+1)}$ by definition. 
The last inequality is due to $\rho(\cdot)$ being non-increasing and $S^\prime_{(t)}\subseteq \hat{S}_{(t)}$. Notice that $\sum_{t=1}^T \rho_{i_t}(\hat{S}_{(t)})=z(\hat{S})$, again by definition of marginal gains. Since $S^\prime$ is a feasible follower solution for $x\in X$, we have
\begin{equation}
	w\geq \Phi(x) \geq z(S^\prime) \geq \sum_{t=1}^T \rho_{i_t}(\hat{S}_{(t)})(1-x_{i_t})= z(\hat{S})-\sum_{t=1}^T \rho_{i_t}(\hat{S}_{(t)})x_{i_t}
\end{equation}
which shows that \eqref{eq:ImprovedCut} is valid for \eqref{eq:SingleL-1}--\eqref{eq:SingleL-4}. It clearly dominates \eqref{eq:basicCut} since $\rho_{i_t}(\hat{S}_{(t)})\leq \rho_{i_t}(\emptyset)$ for each $t\in \{1,...,T\}$.
\end{proof}


{
Next, we propose a method to lift the basic and improved SIC based on the pairwise relationships between some items of the ground set $N$. In a sense, we are informing the model about possible other follower solutions with a better objective value, which can be obtained through the exchange of some items in the current set $\hat{S}$ with \emph{superior} items outside of $\hat{S}$, if the leader does not interdict their usage. To be eligible for this type of exchange, an item pair $(i,j)$ should satisfy the condition that replacing $i$ with $j$ does not increase the marginal gains of the items in $\hat{S}$ with respect to the rest of the set, as well as certain subsets of $\hat{S}$, which implies the superiority of $j$ to $i$. 

We describe how to lift the improved cut \eqref{eq:ImprovedCut} in the following theorem, the proof for the basic cut \eqref{eq:basicCut} works similarly and is omitted for brevity. We also give examples on how the condition mentioned above can occur in the problems considered in the computational study. 

\begin{theorem}

Given a follower set $\hat{S} \in \mathcal{S}$ and an ordering $(i_1,i_2,...,i_T)$ of its elements, let $A=\{a_1,...,a_K\} \subseteq \hat{S}$ and $B=\{b_1,...,b_K\}\subseteq N\setminus\hat{S}$ such that 
\begin{enumerate}[(i)]
\item $c^\ell_{a_k}\geq c^\ell_{b_k}$ for $\ell=1,\ldots ,L$, and
\item $\rho_i(S\cup \{b_k\}\setminus \{a_k\})\leq \rho_i(S)$ for all $S$ such that $(\hat{S}\setminus A)\cup \{a_k\} \subseteq S \subseteq (\hat{S}\cup B) \setminus \{b_k\}$, and $i \in \hat{S}\setminus (S\cup \{b_k\})$,
\end{enumerate}
for each $k \in \{1,...,K\}$. Define the subsets $\hat{S}_{(t)}=\{i_1,...,i_{t-1}\}$ for $T \geq t \geq 2$ and $\hat{S}_{(1)}=\emptyset$. Also define $A_{(k)}=\{a_1,...,a_{k}\}$ and $B_{(k)}=\{b_1,...,b_{k}\}$ for $k\in\{1,...,K\}$, and $A_{(0)}=B_{(0)}=\emptyset$. The following lifted cut is valid for \eqref{eq:SingleL-1}--\eqref{eq:SingleL-4}.
\begin{equation}
w\geq z(\hat{S})-\sum_{t=1}^T \rho_{i_t}(\hat{S}_{(t)})x_{i_t}  + \sum_{k=1}^K \Big(  \rho_{b_k}\big(\hat{S}\cup B_{(k-1)}\big)-\rho_{a_k}\big(\hat{S}\cup \{b_k\}\setminus\{a_k\} \big) \Big)(1-x_{b_k})
	\label{eq:liftCut4}
\end{equation}
\label{theo:liftCut4}
\end{theorem} 
\begin{proof}
Consider any feasible leader solution $x$. If $x_{b_k}=1$ for each $k\in\{1,...,K\}$, the lifted cut is valid since $x$ satisfies the improved cut \eqref{eq:ImprovedCut}. Otherwise, let $\bar{K}=\{k\in \{1,...,K\}:x_{b_k}=0\}$, $A_{\bar{K}}=\{a_k:k\in\bar{K}\}$ and $B_{\bar{K}}=\{b_k:k\in\bar{K}\}$. Define the set $S^\prime=((\hat{S}\setminus A_{\bar{K}})\cup B_{\bar{K}}) \setminus N_x$ which is feasible for $x$ due to condition $(i)$. Since its feasibility implies that $w\geq \Phi(x) \geq z(S^\prime)$, showing that $z(S^\prime)$ is greater than or equal to the RHS of \eqref{eq:liftCut2} would prove the validity of the cut. To this end, we define an intermediate set $S^{\prime\prime}=(\hat{S}\setminus A_{\bar{K}})\cup B_{\bar{K}})$ and compute the following bound on $z(S^{\prime\prime})$ as if at each step we add one $b_k$, $k\in \bar{K}$, to $\hat{S}$ and then remove $a_k$ from the set. 
\begin{equation}
\begin{split}
z(S^{\prime\prime}) &=  z(\hat{S}) + \sum_{k\in \bar{K}}\rho_{b_k}\big(\hat{S}\cup (B_{(k-1)} \cap B_{\bar{K}}) \setminus (A_{(k-1)}\cap A_{\bar{K}}) \big) \\
&  \phantom{aa} - \sum_{k\in \bar{K}}\rho_{a_k}\big(\hat{S}\cup (B_{(k)}\cap B_{\bar{K}}) \setminus (A_{(k)}\cap A_{\bar{K}}) \big) \\
&\geq z(\hat{S}) + \sum_{k\in \bar{K}}\rho_{b_k}\big(\hat{S}\cup B_{(k-1)} \big) - \sum_{k\in \bar{K}}\rho_{a_k}\big(\hat{S}\cup (B_{(k)}\cap B_{\bar{K}}) \setminus (A_{(k)}\cap A_{\bar{K}}) \big)
\end{split}
\label{eq:liftCut_Sprimeprime}
\end{equation}
It is possible to further simplify the RHS of the inequality above as follows by using the condition $(ii)$.
\begin{equation}
\begin{split}
\rho_{a_k} \big(\hat{S}\cup (B_{(k)}\cap B_{\bar{K}})   \setminus  &  (A_{(k)}\cap A_{\bar{K}}) \big) \\
\leq & \, \rho_{a_k}\big(\hat{S}\cup (\{b_2,\ldots,b_{k}\}\cap B_{\bar{K}}) \setminus (\{a_2,\ldots,a_k\}\cap A_{\bar{K}}) \big) \\
\vdots & \\
\leq & \,\, \rho_{a_k}\big(\hat{S}\cup (\{b_{k}\}\cap B_{\bar{K}}) \setminus (\{a_k\}\cap A_{\bar{K}}) \big) \\
\leq & \, \rho_{a_k}\big(\hat{S}\cup\{b_k\}\setminus \{a_k\}\big).
\end{split}
\label{eq:liftCut_exchange}
\end{equation}
The last inequality is due to $k\in \bar{K}$, i.e., $a_k \in A_{\bar{K}}$ and $b_k \in B_{\bar{K}}$. We rewrite the inequality in \eqref{eq:liftCut_Sprimeprime} as 
\begin{equation}
\begin{split}
z(S^{\prime\prime}) &\geq z(\hat{S}) + \sum_{k\in \bar{K}}\rho_{b_k}\big(\hat{S}\cup B_{(k-1)} \big) - \sum_{k\in \bar{K}} \rho_{a_k}\big(\hat{S}\cup\{b_k\}\setminus \{a_k\}\big) \\
&= z(\hat{S}) + \sum_{k=1}^K \Big(\rho_{b_k}\big(\hat{S}\cup B_{(k-1)} \big) -\rho_{a_k}\big(\hat{S}\cup\{b_k\}\setminus \{a_k\}\big) \Big)(1-x_k).
\end{split}
\label{eq:liftCut_Sprimeprime2}
\end{equation}
The equality follows from that $x_k=1$ for $k\notin \bar{K}$. In the next step, we evaluate the objective value of $S^\prime= S^{\prime\prime}\setminus N_x$ using the identity obtained in \eqref{eq:ImprovedCut3}. Let $(j_1,...,j_T)$ be an ordering of the items in $S^{\prime\prime}$, that is identical to $(i_1,...,i_T)$ except that each $a_k$, $k\in \bar{K}$, is replaced with $b_k$, i.e., $j_t=i_t$ for $j_t\in S^{\prime\prime}\setminus B_{\bar{K}}$ and $j_t=b_k \iff i_t=a_k$ for $k\in \bar{K}$. Define the associated subsets $S^{\prime\prime}_{(t)}=\{j_1,\ldots,j_{t-1}\}$. Then, due to \eqref{eq:ImprovedCut3} we have
\begin{equation}
z(S^\prime)\geq z(S^{\prime\prime})-\sum_{t=1}^T \rho_{j_t}(S^{\prime\prime}_{(t)}) x_{j_t} = z(S^{\prime\prime})-\sum_{t=1}^T \rho_{i_t}(S^{\prime\prime}_{(t)}) x_{i_t}.
\label{eq:liftCut_Sprime}
\end{equation}
The reason of the equality is that $x_{j_t}=0$ for $j_t \in B_{\bar{K}}$ and $j_t=i_t$ for  $j_t \in S^{\prime\prime}\setminus B_{\bar{K}}$. Since $S^{\prime\prime}_{(t)}$ is obtained through the exchange of some $a_k$ with $b_k$, condition $(ii)$ implies that $\rho_{i_t}(S^{\prime\prime}_{(t)}) \leq \rho_{i_t}({\hat{S}}_{(t)})$. 
Along with \eqref{eq:liftCut_Sprimeprime2} and \eqref{eq:liftCut_Sprime}, this inequality leads to
\begin{equation}
\begin{split}
z(S^\prime) &\geq z(S^{\prime\prime})-\sum_{t=1}^T \rho_{i_t}({\hat{S}}_{(t)}) x_{i_t} \\
&\geq   
z(\hat{S}) + \sum_{k=1}^K \Big(\rho_{b_k}\big(\hat{S}\cup B_{(k-1)} \big) -\rho_{a_k}\big(\hat{S}\cup\{b_k\}\setminus \{a_k\}\big) \Big)(1-x_k) -\sum_{t=1}^T \rho_{i_t}({\hat{S}}_{(t)}) x_{i_t} 
\end{split}
\end{equation}
which completes the proof.
\end{proof}

\begin{remark}
For a pair $(a_k,b_k)$ satisfying condition $(ii)$, it is possible that the coefficient $\rho_{b_k}\big(\hat{S}\cup B_{(k-1)} \big)-\rho_{a_k}\big(\hat{S}\cup \{b_k\}\setminus\{a_k\} \big)$ of the last term in \eqref{eq:liftCut4} is negative. 
Even if the items are selected in such a way that $b_k$ can replace $a_k$ without any sacrifice in solution quality, the coefficient can still be negative since its components constitute bounds on the true change in the objective due to adding $b_k$ and removing $a_k$, respectively. 
Obviously, \eqref{eq:liftCut4} dominates the improved cut \eqref{eq:ImprovedCut} only if the coefficients of the $(1-x_{b_k})$ terms are non-negative. Since one can choose the sets $A$ and $B$ accordingly, we assume that those coefficients are positive and call \eqref{eq:liftCut4} a \textit{lifted cut}.
	\label{rem:lifting}
\end{remark}

\begin{remark}

Condition $(ii)$ in Theorem \ref{theo:liftCut4} is a general description for the implication of superiority between item pairs. 
Although this type of relationship might seem difficult to detect, it becomes more intuitive on an application basis. 
For \BIIG, $(ii)$ corresponds to the situation that the neighbors (connected targets) of $a_k$ is a subset of the neighbors of $b_k$, and $p_{b_k}\geq p_{a_k}$. Thus, if $a_k$ is replaced by $b_k$, the objective is at least as large as before and the marginal gains of $i\in \hat{S}$ with respect to rest of the set is not larger than before.
In some problems such as \WMCIG superiority has stronger implications. There, for a facility pair $(a_k,b_k)$ with $J(a_k)\subseteq J(b_k)$ condition $(ii)$ is equivalent to $\rho_{a_k}(\{b_k\})=0$. In other words including $b_k$, which covers all of the customers that $a_k$ covers, in the set renders $a_k$ completely useless and removing $a_k$ does not damage the objective anymore. 
Note that, condition $(ii)$ holds for any pair with $\rho_{a_k}(\{b_k\})=0$. If $\rho_{a_k}(\{b_k\})=0$ for each $k\in\{1,...,K\}$, then \eqref{eq:liftCut4} is reduced to 
	\begin{equation}
	w\geq z(\hat{S})-\sum_{t=1}^T \rho_{i_t}(\hat{S}_{(t)})x_{i_t} + \sum_{k=1}^K \rho_{b_k}(\hat{S}\cup B_{(k-1)} )(1-x_{b_k}).
	\label{eq:liftCut2}
	\end{equation}
This cut dominates \eqref{eq:ImprovedCut} independent of the choices of $A$ and $B$ since the last term is always non-negative.
	\label{rem:liftSpecial}
\end{remark}

\begin{example}

Consider an instance of the \BIIG with 3 items $N=\{1,2,3\}$, 4 targets $M=\{a,b,c,d\}$ and the arc list $A=\{(1,a),(2,a),(2,b),(3,a),(3,c)\}$. Let the activation probabilities of the items be given as $p_1=0.3$, $p_2=0.5$ and $p_3=0.4$. Recall that the objective function of the follower is $z(S)=\sum_{j\in M} \big( 1- \prod_{i\in S: (i,j)\in A} (1-p_i) \big)$, therefore $z(\emptyset)=0$ and the gains with respect to the empty set are $\rho_1(\emptyset)=0.3+0+0=0.3$, $\rho_2(\emptyset)=0.5 + 0.5+0=1$ and $\rho_3(\emptyset)=0.4 +0+ 0.4=0.8$ Consider the set $\hat{S}=\{1,2\}$ with objective value $z(\hat{S})=1-(1-0.3)(1-0.5)+ 0.5 +0= 1.15$. The associated basic SIC is $w \geq 1.15 - 0.3x_1 - x_2$. If we use the ordering $i_1=1,i_2=2$, the improved cut becomes $w\geq 1.15 - 0.3x_1 -0.85x_2$. 

Now consider the items $1 \in \hat{S}$ and $3 \notin \hat{S}$. Notice that they do not satisfy the special condition in Remark \ref{rem:liftSpecial}, i.e, $\rho_1 (\{3\})= z(\{1,3\})-z(\{3\}) \neq 0$. For the condition $(ii)$ in Theorem \ref{theo:liftCut4}, it is sufficient to check if $\rho_2(\{3\})\leq \rho_2(\{1\})$ since only $i=2$ and $S=\{1\}$ fit the definition given. We have that $\rho_2(\{3\})=z(\{2,3\})-z(\{3\})=1-(1-0.5)(1-0.4)+0.5 +0.4 - 0.8=1.6-0.8=0.8 $ and $\rho_2(\{1\})=z(\{1,2\})-z(\{1\})=1.15-0.3=0.85$, thus the condition is satisfied and a better solution can be found by replacing 1 with 3. For $A=\{1\}$ and $B=\{3\}$ the coefficient of the lifting term is $\rho_3(\{1,2\})-\rho_1(\{2,3\})=z(\{2,3\})-z(\{1,2\})=1.6-1.15=0.45$. The improved cut is lifted to $w \geq 1.15 -0.3x_1-0.85x_2 + 0.45 (1-x_3) $ according to Theorem \ref{theo:liftCut4}.

\end{example}

\begin{example}
Consider an instance of \WMCIG with 4 customers $J=\{a,b,c,d\}$ and 3 potential facility locations $N=\{1,2,3\}$ and the maximum number of facilities to open $B=2$. Let the customers covered by each location are given as $J(1)=\{a,c\}$, $J(2)=\{a,b\}$, $J(3)=\{a,c,d\}$, and the profits of the customers are $p_a=5$, $p_b=9$, $p_c=6$ and $p_d=4$. Now consider the set $\hat{S}=\{1,2\}$ with a total profit $z(\hat{S})=5+9+6=20$. The basic SCI for $\hat{S}$ is unique and $w\geq 20 - 11x_1 - 14x_2$. Since the size of $\hat{S}$ is two, there are two ways to generate the improved cut: $w \geq 20 - 11x_1 - 9x_2$ and $w \geq 20 -6x_1-14x_2$. The facility locations $1 \in \hat{S}$ and $3 \notin \hat{S}$ satisfy the condition in Remark \ref{rem:liftSpecial} as $\rho_1(\{3\})= z(\{1,3\})-z(\{3\})=0$. If we define the sets $A=\{1\}$ and $B=\{3\}$, using $\rho_3(\{1,2\})=4$ the first improved cut can be lifted to $w \geq 20 - 11x_1 - 9x_2 + 4(1-x_3)$.
Similarly, the second one is lifted to $w \geq 20 -6x_1-14x_2+ 4(1-x_3)$. If $x_3=1$, the lifted cuts are identical to their non-lifted versions, otherwise they yield a better cut.
\label{example:WMCIG}
\end{example}


\begin{remark}
If the term $\rho_{i_t}(\hat{S}_{(t)})$ in the lifted cut is replaced by $\rho_{i_t}(\emptyset)$, the resulting cut is the lifted version of the basic SIC \eqref{eq:basicCut} and it is obviously valid and dominated by the current one since $\rho_{i_t}(\hat{S}_{(t)}) \leq \rho_{i_t}(\emptyset)$. 
\end{remark}

Next, we propose a method to obtain a new cut that could feed the model with an alternative follower solution set, in case some of the items in the current set are interdicted. Unlike the lifted cut, this cut is not based on superiority implying relationships between item pairs, but on the availability of the items.

}
\begin{theorem}

Given a follower set $\hat{S} \in \mathcal{S}$ and an ordering $(i_1,i_2,...,i_T)$ of its elements, let $A=\{a_1,...,a_K\} \subseteq \hat{S}$ and $B=\{b_1,...,b_K\}\subseteq N\setminus\hat{S}$ such that $c^\ell_{a_k}\geq c^\ell_{b_k}$ for $\ell=1,\ldots ,L$ and for each $k\in \{1,...,K\}$. Define the subsets $\hat{S}_{(t)}=\{i_1,...,i_{t-1}\}$ for $T \geq t \geq 2$ and $\hat{S}_{(1)}=\emptyset$. Also define $B_{(k)}=\{b_1,...,b_{k}\}$ for $k\in\{1,...,K\}$ and $B_{(0)}=\emptyset$. The following alternative cut is valid for \eqref{eq:SingleL-1}--\eqref{eq:SingleL-4}. 

\begin{equation}
	w\geq z(\hat{S})-\sum_{t=1}^T \rho_{i_t}(\hat{S}_{(t)})x_{i_t} + \sum_{k=1}^K \rho_{b_k}(\hat{S}\cup B_{(k-1)}\setminus \{a_k\})(x_{a_k}-x_{b_k})
	\label{eq:modCut2}
\end{equation}
\label{theo:modCut2}
\end{theorem}

\begin{proof}

Let $x$ be a feasible leader solution. If $x_{a_k}-x_{b_k}\leq 0$ for all $k\in \{1,...,K\}$, then $x$ satisfies \eqref{eq:modCut2} since it already satisfies \eqref{eq:ImprovedCut} since $\rho(\cdot)$ is non-negative by \eqref{eq:Nemhauser3}. Otherwise, let $\bar{K}=\{k\in \{1,...,K\}:x_{a_k}-x_{b_k}>0\}$ denote the index set of the $a_k,b_k$ pairs such that $x_{a_k}=1$ and $x_{b_k}=0$.
Define $A_{\bar{K}}=\{a_k:k\in\bar{K}\}$ and $B_{\bar{K}}=\{b_k:k\in\bar{K}\}$. 
Consider the set $S^\prime=((\hat{S}\setminus A_{\bar{K}})\cup B_{\bar{K}}) \setminus N_x$ which is feasible for $x$ under the assumption that $c^\ell_{a_k}\geq c^\ell_{b_k}$ for $\ell=1,\ldots ,L$ and for each $k\in \bar{K}$. Due to the definition of $A_{\bar{K}}$ and $B_{\bar{K}}$, we have that $A_{\bar{K}}\subseteq N_x$ and $B_{\bar{K}}\cap N_x=\emptyset$. Thus, $S^\prime=(\hat{S}\setminus N_x) \cup B_{\bar{K}}$ and a lower bound on its objective value can be obtained by estimating the incremental change in the objective value due to adding each $b_k\in B_{\bar{K}}$ to $\hat{S}\setminus N_x$ as follows.
\begin{equation}
\begin{split}
z(S^\prime) & = z(\hat{S}\setminus N_x) + \sum_{k\in\bar{K}}\rho_{b_k} (\hat{S}\setminus N_x \cup (B_{(k-1)}\cap B_{\bar{K}}))\\
& \geq z(\hat{S}\setminus N_x) +  \sum_{k\in\bar{K}} \rho_{b_k} (\hat{S}\setminus a_k \cup B_{(k-1)})\\
& \geq z(\hat{S}\setminus N_x) +  \sum_{k=1}^K \rho_{b_k} (\hat{S}\setminus a_k \cup B_{(k-1)}) (x_{a_k}-x_{b_k}).
\end{split}
\label{eq:modCut2_Sprime}
\end{equation}
The reason of the first inequality is that $a_k \in N_x$ for $k\in \bar{K}$ and $\rho(\cdot)$ is non-increasing. The second inequality follows from that $x_{a_k}-x_{b_k}=1$ for $k\in \bar{K}$ and $x_{a_k}-x_{b_k}\leq 0$ for $k\notin \bar{K}$. Due to \eqref{eq:ImprovedCut3} in the proof of Theorem \ref{theo:ImprovedCut} we have that $z(\hat{S}\setminus N_x) \geq z(\hat{S})-\sum_{t=1}^T \rho_{i_t}(\hat{S}_{(t)})x_{i_t}$. Thus, $z(\hat{S}\setminus N_x)$ in \eqref{eq:modCut2_Sprime} can be replaced by its lower bound yielding
\begin{equation}
w\geq \Phi(x) \geq z(S^\prime) \geq z(\hat{S})-\sum_{t=1}^T \rho_{i_t}(\hat{S}_{(t)})x_{i_t} + \sum_{k=1}^K \rho_{b_k} (\hat{S}\setminus a_k \cup B_{(k-1)}) (x_{a_k}-x_{b_k})
\end{equation}
and hereby completing the proof.

\end{proof}

Theorem \ref{theo:modCut2} can be interpreted as follows. If an item in $\hat{S}$ is interdicted and therefore removed from the solution, one could obtain a better solution by including a non-interdicted item whose costs are at most as large as of the former item. However, the new cut can be worse than the original one as $x_{a_k}-x_{b_k}$ can take a negative value. Thus, \eqref{eq:modCut2} is not necessarily a lifted cut, but an alternative to the original basic/improved SIC. As is the case with the lifted cuts, an alternative cut can be obtained from a basic cut \eqref{eq:basicCut} instead of an improved one, by simply replacing $\rho_{i_t}(\hat{S}_{(t)})$ in \eqref{eq:modCut2} by $\rho_{i_t}(\emptyset)$. However, it would be a weaker cut than \eqref{eq:modCut2}.

\begin{example}

Consider the \WMCIG instance in Example \ref{example:WMCIG}. Suppose that we are given the same set $\hat{S}=\{1,2\}$ and asked to obtain the alternative cut for the ordering $i_1=1, i_2=2$ and $A=\{1\}$, $B=\{3\}$. The coefficient of the additional term would be $\rho_3(\hat{S}\setminus \{1\})=\rho_3(\{2\})=24-14=10$. Using the improved cut from Example \ref{example:WMCIG}, we obtain the alternative cut
\begin{equation}
w\geq 20-11x_1-9x_2 + 10 (x_1-x_3). \notag
\end{equation}
Now consider two interdiction strategies $x^{(1)}=(0,0,1)$ and $x^{(2)}=(1,0,0)$. While $x^{(1)}$ yields an alternative cut worse than the improved one ($w\geq 10$ instead of $w\geq 20$), the alternative cut is better for $x^{(2)}$ ($w\geq 19$ instead of $w\geq 9$).
\end{example}

\section{Implementation Detail \label{section:implementation}}
In this section we propose a branch-and-cut (B\&C) scheme to solve Problem \eqref{eq:P1}. We explain the details of the separation of SICs after we provide some observations which can be exploited for a more efficient implementation.

\subsection{Dominance Inequalities}

The following dominance inequalities can be added to remove some feasible solutions, but it is guaranteed that not all optimal solutions will be cut off.  
	Detection of the item-pairs fitting Theorem \ref{theo:dominanceValidity} depends on the problem structure. Thus, after the main theorem, we give propositions on how to detect them in our considered applications.

\begin{theorem}
If a pair of items $i,j \in N$ satisfies $c^\ell_i \leq c^\ell_j$ for $\ell=1,\ldots, L$, $\rho_{i}(S)\geq \rho_{j}(S)$ for all $S\subseteq N\setminus \{i,j\}$ and $A_i \leq A_j$, then the inequality $x_i \geq x_j $ does not cut off all optimal solutions to \eqref{eq:SingleL-1}--\eqref{eq:SingleL-4} 
if $x_j \geq x_i$ is not already present in the model. 
\label{theo:dominanceValidity} 
\end{theorem}
\begin{proof}
Suppose that all optimal solutions are eliminated by dominance inequalities and $x^*$ is one of them. For the sake of simplicity, assume that the model includes exactly one such inequality, which is $x_i \geq x_j$. Since $x^*$ is cut by the dominance inequality, we should have $x^*_i<x^*_j$, i.e., $x^*_i=0$ and $x^*_j=1$. Now define $x^\prime$ identical to $x^*$ except that $x^\prime_i=1$, $x^\prime_j=0$. Since $A_i \leq A_j$, $x^\prime$ is feasible, and it also satisfies the dominance inequality. Now, let $S^\prime$ be an optimal follower response to $x^\prime$. 
If $j\notin S^\prime$, then $S^\prime$ is also feasible for $x^*$, and $\Phi(x^*)\geq z(S^\prime)=\Phi(x^\prime)$. Otherwise, $(S^\prime \setminus \{j\}) \cup \{i\}$ is feasible for $x^*$ due to the assumption $c^\ell_i \leq c^\ell_j$ for $\ell=1,\ldots, L$. 
Moreover, since $\rho_{i}(S)\geq \rho_{j}(S)$ for all $S\subseteq N\setminus \{i,j\}$ we have $\rho_i(S^\prime \setminus \{j\})\geq \rho_j(S^\prime \setminus \{j\})$, which implies by definition of $\rho(\cdot)$ that
\begin{equation}
z((S^\prime \setminus \{j\}) \cup \{i\})\geq z(S^\prime).
\end{equation}
As a result, $\Phi(x^*)\geq z((S^\prime \setminus \{j\}) \cup \{i\}) \geq z(S^\prime)=\Phi(x^\prime)$, and it shows that $x^\prime$ is also optimal. 
In case the model includes $n>1$ dominance inequalities $x_{i_k}\geq x_{j_k}$, $k\in\{1,\ldots, n\}$, the same procedure applies: for each violated inequality $k$ set $x^\prime_{i_k}=1$ and $x^\prime_{j_k}=0$, and the same result follows.

\end{proof}


\begin{proposition}
Given a non-decreasing and monotone function $z$, a ground set $N$ and $i,j\in N$, if $z(\{i,j\})=\rho_i(\emptyset)$, then $\rho_{i}(S)\geq \rho_{j}(S)$, $\forall S\subseteq N$.
\label{prop:dominanceMaxCover}
\end{proposition}

\begin{proof}
Given that $z$ is monotone, i.e., $z(\emptyset)=0$, $z(\{i,j\})=\rho_i(\emptyset)+\rho_j(\{i\})$. If $z(\{i,j\})=\rho_i(\emptyset)$, then $\rho_j(\{i\})=\rho_j(\{i\}\cup S)=0$ for all $S\subseteq N$. We need to show that $\rho_{i}(S)- \rho_{j}(S)\geq 0$. By definition of marginal gains we have $\rho_{i}(S)- \rho_{j}(S)=z(S\cup\{i\})-z(S\cup \{j\})$. Using the submodular inequality \eqref{eq:Nemhauser1} we can write
\begin{equation}
z(S\cup\{j\})\leq z(S\cup \{i\}) + \rho_j(S\cup \{i\}) - \rho_i(S\cup \{j\}).
\end{equation}
Since $\rho_j(S \cup\{i\})=0$, the inequality becomes
\begin{equation}
z(S\cup \{i\}) - z(S\cup\{j\})\geq \rho_i(S\cup \{j\})\geq 0,
\end{equation}
which completes the proof.
\end{proof}

\begin{proposition}
Given an instance of the \BIIG for the bipartite graph $G=(N,M,A)$, let $M(i)=\{k\in M:(i,k)\in A\}$ denote the target set of each item $i\in N$. Let a pair of items $i,j\in N$ satisfy $M(j) \subseteq M(i)$ and $p_i\geq p_j$. Then, $\rho_{i}(S)\geq \rho_{j}(S)$, $\forall S\subseteq N\setminus \{i,j\}$.
\label{prop:dominanceBipartite}
\end{proposition}

\begin{proof}
Using the definition of $z$ under \BIIG, we have for each $S\in N\setminus \{i,j\}$ that 
\begin{equation}
\begin{split}
\rho_i(S) & = z(S\cup \{i\})-z(S) = \sum_{k\in M} \big( 1- \prod_{i^\prime\in S\cup \{i\}: (i^\prime,k)\in A} (1-p_{i^\prime}) \big) - \sum_{k\in M} \big( 1- \prod_{i^\prime\in S: (i^\prime,k)\in A} (1-p_{i^\prime}) \big) \\
&= \sum_{k\in M(i)} \big(  \prod_{i^\prime\in S: (i^\prime,k)\in A} (1-p_{i^\prime})  - \prod_{i^\prime\in S\cup \{i\}: (i^\prime,k)\in A} (1-p_{i^\prime})  \big) \\
& \phantom{aa} + \sum_{k\in M\setminus M(i)} \big(  \prod_{i^\prime\in S: (i^\prime,k)\in A} (1-p_{i^\prime})  - \prod_{i^\prime\in S\cup \{i\}: (i^\prime,k)\in A} (1-p_{i^\prime})  \big) \\
& = \sum_{k\in M(i)} \prod_{i^\prime \in S:(i^\prime,k)\in A}(1-p_{i^\prime})p_i 
\geq \sum_{k\in M(j)} \prod_{i^\prime \in S:(i^\prime,k)\in A}(1-p_{i^\prime})p_j = \rho_j(S)
\end{split}
\end{equation}
The reason of the first equality in the last line is that for each $k\in M\setminus M(i)$ the two products are identical, i.e., $\{i^\prime \in S:(i^\prime,k)\in A\}=\{i^\prime \in S\cup \{i\}:(i^\prime,k)\in A\}$. To put it simply, only the activation probabilities of $k\in M(i)$ are affected due to adding $i$ to $S$. The inequality follows from the assumptions $M(j) \subseteq M(i)$ and $p_i\geq p_j$, in addition to all terms in the product being non-negative.
\end{proof}

In our implementation, for \WMCIG instances we use the condition in Proposition \ref{prop:dominanceMaxCover} to detect pairs that fit into the description in Theorem \ref{theo:dominanceValidity}, due to problem characteristics discussed in Remark \ref{rem:liftSpecial}. For \BIIG on the other hand, we check the condition in Proposition \ref{prop:dominanceBipartite} for each pair. We add the resulting dominance inequalities to the initial model. If the conditions are fulfilled in both directions, then the items can substitute each other, and only one of the resulting inequalities is used.

\subsection{Maximal Follower Solutions}

A follower solution $\hat{S}\in \mathcal{S}$ is called maximal if there is no $S^\prime \in \mathcal{S} $ such that $\hat{S}\subset S^\prime $. \citet{fischetti2019interdiction} consider only maximal follower solutions while separating their interdiction cuts since for their setting they showed that their proposed interdiction cut for a maximal solution $\hat{y}$ dominates the one for $y^\prime <\hat{y}$. 
\begin{theorem}
Consider a maximal follower solution $\hat{S}\in\mathcal{S}$ and $S^\prime \subset \hat{S}$. The basic interdiction cut \eqref{eq:basicCut} for $\hat{S}$ does not dominate the one for $S^\prime$.
\label{prop:maximal}
\end{theorem}

\begin{proof}

Suppose that the basic interdiction cut for $\hat{S}$ dominates the one for $S^\prime$. Then the RHS of the basic cut \eqref{eq:basicCut} for $S^\prime$ should be less than or equal to the RHS of the cut for $\hat{S}$, for all $x \in X$. Subtracting the former from the latter yields
\begin{equation}
\begin{split}
 & z(\hat{S})- \sum_{i\in \hat{S} }\rho_i(\emptyset)x_i -  z(S^\prime) +\sum_{i\in S^\prime}\rho_i(\emptyset)x_i = z(\hat{S})-z(S^\prime)- \sum_{i\in \hat{S} \setminus S^\prime}\rho_i(\emptyset)x_i \\
&\leq \sum_{i\in \hat{S} \setminus S^\prime}\rho_i(S^\prime)- \sum_{i\in \hat{S} \setminus S^\prime}\rho_i(\emptyset)x_i  = \sum_{i\in \hat{S} \setminus S^\prime}\big(\rho_i(S^\prime)-\rho_i(\emptyset)x_i \big).
\end{split}
\end{equation}
The inequality sign comes from the submodular inequality \eqref{eq:Nemhauser4}. Consider the case that $x_i=1$ for $i\in \hat{S} \setminus S^\prime$. The difference will be non-positive since $\rho_i(S^\prime)\leq \rho_i(\emptyset)$, which is a contradiction.

\end{proof}

\begin{theorem}
Consider a maximal follower solution $\hat{S}\in\mathcal{S}$ with the ordering $(i_1,...,i_T)$ and $S^\prime\subset \hat{S}$ with ordering $(i_1,...,i_{T-k})$, where $T>k>0$. The improved interdiction cut \eqref{eq:ImprovedCut} for $\hat{S}$ dominates the one for $S^\prime$.
\label{prop:maximal2}
\end{theorem}

\begin{proof}
We need to show that the RHS of \eqref{eq:ImprovedCut} for $S^\prime$ is less than or equal to the RHS of the cut for $\hat{S}$ when the given orderings are used to generate the cuts. Define $\hat{S}_{(t)}=S^\prime_{(t)}=\{i_1,\ldots,i_{t-1}\}$ for $T \geq t \geq 2$, and $\hat{S}_{(1)}=S^\prime_{(1)}=\emptyset$, for the sake of better notation, even though not all $S^\prime_{(t)}$ are subsets of $S^\prime$. Then, the difference which needs to be proven non-negative is 
\begin{equation}
z(\hat{S})-\sum_{t=1}^T \rho_{i_t}(\hat{S}_{(t)})x_{i_t} - z(S^\prime) + \sum_{t=1}^{T-k} \rho_{i_t}(S^\prime_{(t)})x_{i_t} =  z(\hat{S}) - z(S^\prime) -\sum_{t=T-k+1}^T \rho_{i_t}(\hat{S}_{(t)})x_{i_t}.
\end{equation}
Since $\hat{S}\setminus S^\prime=\{i_{T-k-1},\ldots,i_T\}$, the difference $z(\hat{S}) - z(S^\prime)$ can be computed exactly and the above expression is rewritten as follows.
\begin{equation}
\sum_{t=T-k+1}^T \rho_{i_t} (S^\prime_{(t)}) -\sum_{t=T-k+1}^T \rho_{i_t}(\hat{S}_{(t)})x_{i_t} =  \sum_{t=T-k+1}^T \rho_{i_t} (\hat{S}_{(t)}) (1-x_{i_t})\geq 0
\end{equation}

\end{proof}

Theorem \ref{prop:maximal2} indicates that replacing a non-maximal follower solution with a maximal one by appending new items to it without disrupting the initial ordering yields a better improved cut. On the other hand, if the improved cuts for $\hat{S}$ and $S^\prime\subset \hat{S}$ are generated based on arbitrary orderings of their elements, it is not possible to claim that one cut is better than the other since the value of the differences on the right-hand-side depends on $x$.

\subsection{Separation of Basic and Improved Submodular Interdiction Cuts}
\label{section:Sep_basic}

We have different separation procedures for SICs for integer and fractional solutions $x^*$ encountered in our B\&C. We first discuss the separation of integer solutions and then the separation of fractional solutions.

\subsubsection{Separation of Integer Solutions.}
\label{section:Sep_basic_integer}
Given a leader solution $x^*\in X$, let $N_{\text{-}x}=N\setminus N_x=\{i\in N: x_i=0\}$ be the set of items available to the follower and $y\in \{0,1\}^{|N_{\text{-}x}|}$ denote the characteristic vector of any follower solution $S\subseteq N_{\text{-}x}$, i.e., $S=\{i\in N_{\text{-}x}: y_i=1\}$. As a result of submodular inequalities \eqref{eq:Nemhauser1} and \eqref{eq:Nemhauser2}, the follower's problem can be formulated as an MIP as follows \citep{ahmed2011maximizing}, which yields our separation problem $(SEP)$.
\begin{align}
(SEP) \quad	\Phi(x)= &\max \theta \label{eq:sepProblem1} \\
	&\text{s.t.} \notag\\
	& \theta \leq z(\hat{S}) +\sum_{i\in N_{\text{-}x}\setminus \hat{S}} \rho_i(\hat{S})y_i  - \sum_{i\in \hat{S}}\rho_i(N_{\text{-}x} \setminus \{i\})(1-y_i) &\hat{S}& \subseteq N_{\text{-}x} \label{eq:sepProblem2}\\
	& \theta \leq z(\hat{S}) +\sum_{i\in N_{\text{-}x}\setminus \hat{S}} \rho_i(\emptyset)y_i  - \sum_{i\in \hat{S}}\rho_i(\hat{S} \setminus \{i\})(1-y_i) &\hat{S}& \subseteq N_{\text{-}x} \label{eq:sepProblem3} \\
	& \sum_{i\in N}c^\ell_i y_i\leq Q_\ell  &\ell&\in \{1,\ldots, L\}\label{eq:sepProblem4}\\
	& y_i\in \{0,1\} &i&\in N_{\text{-}x} 	\label{eq:sepProblem5}
\end{align}

\paragraph{Solving the separation problem $(SEP)$.} 
The separation problem $(SEP)$ can be solved via a branch-and-cut scheme where submodular cuts \eqref{eq:sepProblem2} and \eqref{eq:sepProblem3} are generated as they are needed. To this end, the formulation composing of \eqref{eq:sepProblem1}, \eqref{eq:sepProblem4}, and \eqref{eq:sepProblem5} is solved using an MILP solver. Let $(\theta^*,y^*)$ be the solution of the LP at the current (follower) B\&C tree node. If $y^*$ is integer feasible, then it defines a unique set $\hat{S}=\{i\in N_{\text{-}x}: y^*_i=1\}$ and the value of $z(\hat{S})$ is computed according to the definition of $z$. If $\theta^*> z(\hat{S})$, then \eqref{eq:sepProblem2} and \eqref{eq:sepProblem3} are generated for $\hat{S}$ (by evaluating the necessary marginal gains); otherwise no cut is added. For fractional $y^*$, $\hat{S}$ is obtained in a heuristic way as follows. $\hat{S}$ is initialized as an empty set, the items $i \in N_{\text{-}x}$ are sorted in non-increasing order of $y_i$ values, and they are added to $\hat{S}$ in this order until a knapsack constraint \eqref{eq:sepProblem4} is violated. \eqref{eq:sepProblem2} and \eqref{eq:sepProblem3} are obtained for $\hat{S}$ and their amounts of violation are computed for $(\theta^*,y^*)$. The violated cuts are added to the problem, if any. Notice that the $\rho_i(N_{\text{-}x} \setminus \{i\})$ values in \eqref{eq:sepProblem2} are independent of $\hat{S}$, and can be approximated by $\rho_i(N \setminus \{i\})$. We compute them once in advance, instead of calculating each time the follower's problem is solved. The cut \eqref{eq:sepProblem2} is still valid since $\rho_i(N \setminus \{i\}) \leq \rho_i(N_{\text{-}x} \setminus \{i\})$ for any $x$. Also note that, while solving the subproblem of \BIIG we make use of the greedy fractional separation proposed in \cite{salvagnin2019some}. 

\paragraph{The separation procedure.} 
Now, let $(w^*,x^*)$ be the optimal solution at the current B\&C node with $(w^*,x^*)$ integer. The SICs are separated exactly as follows. First, the separation problem $(SEP)$
is solved on $N_{\text{-}x^*}$ as described above to obtain $\hat{S}$, which is defined by its optimal solution, and its objective value $z(\hat{S})$. If $w^*<z(\hat{S})$, then $\hat{S}$ yields a violated basic \eqref{eq:basicCut} or improved \eqref{eq:ImprovedCut} SIC. 
While the coefficients in a basic cut are independent of $\hat{S}$ and only computed once as a pre-processing step, the coefficients of an improved cut depend on $\hat{S}$ and require an ordering of its elements. For the latter, the items in $\hat{S}$ are sorted in non-increasing order of $\rho_i(\emptyset)$ values, which performs better than non-decreasing or random ordering in our preliminary experiments.

Although we need to separate integer solutions exactly for the correctness of our algorithm, it is also possible to first try to use a heuristic method to find a violated cut instead of solving the separation problem to optimality to potentially speed up the separation. We propose an \emph{enhanced} exact separation procedure as an alternative to the method described above. We first implement a classical greedy algorithm which is given as Algorithm \ref{Alg:greedy} to find a feasible follower solution $\hat{S}$. If $\hat{S}$ leads to violated SIC, then we are done. Otherwise, we solve the separation problem $(SEP)$ with a B\& C until a desired solution is reached. The procedure is summarized in Algorithm \ref{Alg:IntSepHeur} which returns $\hat{S}$ yielding a violated SIC, if there exists one. Otherwise, it returns an empty set which implies that the current solution is the new incumbent. The ordering for the improved cut is obtained as before. 

\begin{algorithm}[h]
\caption{$Greedy(N,\hat{S}, O)$}
\begin{algorithmic}[1]
\WHILE{$\exists i \in N \setminus \hat{S}$ such that $c(\hat{S}\cup \{i\})\leq Q$ }
\STATE{ $i^* \leftarrow \arg \max_{i\in N \setminus \hat{S} : c(\hat{S}\cup \{i\})\leq Q } z(\hat{S} \cup \{i\})   $}
\STATE{$\hat{S}\leftarrow \hat{S} \cup \{i^*\}$, $O.add(i^*)$}
\ENDWHILE
\STATE{Return $(\hat{S}, O)$}
\end{algorithmic}
\label{Alg:greedy}
\end{algorithm}

\begin{algorithm}[h]
\caption{Enhanced Separation of Integer Solutions}
{\textbf{Input:} An integer feasible leader solution $(w^*,x^*)$}\\
{\textbf{Output:} A follower solution $\hat{S} \in \mathcal{S}$}
\begin{algorithmic}[1]
\STATE{$\hat{S}\leftarrow \emptyset$, $N_{\text{-}x}=\{i\in N: x^*_i=0\}$, $O=()$}
\STATE{$(\hat{S},O)\leftarrow Greedy(N_{\text{-}x}, \hat{S}, O)$}
\IF{the SIC defined by $\hat{S}$ is not violated at $(w^*,x^*)$}
\STATE{Solve the separation problem until a feasible solution $y^*$ with objective $\theta^*> w^*$ is found}
\IF{a solution is found}
\STATE{$\hat{S} \leftarrow \{i\in N_{\text{-}x}: y^*_i=1\}$}
\ELSE
\STATE{There is no violated SIC, $\hat{S}\leftarrow \emptyset$}
\ENDIF
\ENDIF
\STATE{Return $\hat{S}$}
\end{algorithmic}
\label{Alg:IntSepHeur}
\end{algorithm}

\subsubsection{Separation of Fractional Solutions.}
\label{section:Sep_basic_frac}

For fractional $x^*$, $\hat{S}$ is obtained in a greedy way. If the relative violation of the resulting cut exceeds the threshold of 1\%, then it is added to the problem. 
The following three options are considered to obtain $\hat{S}$ and the ordering of its elements which is required for the improved cut. 

\begin{itemize}
\item S1 (see Algorithm \ref{Alg:S1}): First a temporary ground set is determined by using only totally non-interdicted items, i.e., $i:x^*_i=0$, and then the $Greedy(\cdot)$ function is called. While generating the basic cut, non-maximal follower solutions are used 
If the cut to be generated is an improved cut and the solution is not maximal, the $Greedy(\cdot)$ function is re-invoked to reach a maximal solution. This is done according to Theorem \ref{prop:maximal2}, i.e., items are appended to the end of the current ordering by $O.add(\cdot)$.
\item S2 (see Algorithm \ref{Alg:S2}): The same procedure used for S1 is followed except the definition of the ground set. Here, it is obtained based on a rounding of $x^*$, which allows to include some items with fractional $x^*_i$ in the ground set. 
\item S3 (see Algorithm \ref{Alg:S3}): Another greedy approach is used to obtain a maximally violated basic/improved cut. The violation increase due to adding an item $i\in N\setminus \hat{S}$ to $\hat{S}$ is denoted by $v_i(\hat{S})$ and evaluated by $\rho_i(\hat{S})-\rho_i(\emptyset)x^*_i$ for the basic cut and by $\rho_i(\hat{S})(1-x^*_i)$ for the improved cut. The item with the maximum $v_i(\hat{S})$ is added to the set until the budget is reached or the maximum $v_i(\hat{S})$ is negative (only possible for the basic cut).
\end{itemize}

\begin{algorithm}[h]
\caption{S1 }
{\textbf{Input:} A fractional leader solution $x^*$}\\
{\textbf{Output:} A follower solution $\hat{S} \in \mathcal{S}$ and an ordering of its elements}
\begin{algorithmic}[1]
\STATE{$\hat{S}\leftarrow \emptyset$, $N_{\text{-}x}=\{i\in N: x^*_i=0\}$, $O=()$}
\STATE{$(\hat{S},O)\leftarrow Greedy(N_{\text{-}x}, \hat{S}, O)$}
\IF{${cutType}={Improved}$ and $\exists i \in N \setminus \hat{S}$ such that $c(\hat{S}\cup \{i\})\leq Q$ }
\STATE{$(\hat{S},O)\leftarrow Greedy(N, \hat{S}, O)$}
\ENDIF
\STATE{Return $\hat{S}$ and ordering $O=(i_1, \ldots, i_{|\hat{S}|})$}
\end{algorithmic}
\label{Alg:S1}
\end{algorithm}

\begin{algorithm}[h]
\caption{S2 }
{\textbf{Input:} A fractional leader solution $x^*$}\\
{\textbf{Output:} A follower solution $\hat{S} \in \mathcal{S}$ and an ordering of its elements}
\begin{algorithmic}[1]
\STATE{$\hat{S}\leftarrow \emptyset$, $O=()$, $x^\prime\leftarrow 0$}
\FOR{each $i:x^*_i=1$ }\STATE{$x^\prime_i\leftarrow 1$}
\ENDFOR
\WHILE{$\exists i \in N:\, x^\prime_i=0, A(x^\prime +e_i)\leq b$}
\STATE{$i^\prime \leftarrow \arg \max_{i\in N:\, x^\prime_i=0, A(x^\prime +e_i)\leq b}  x^*_i  $}
\STATE{$x^\prime_i \leftarrow 1$}
\ENDWHILE
\STATE{$N_{\text{-}x}=\{i\in N: x^\prime_i=0\}$}
\STATE{$(\hat{S},O)\leftarrow Greedy(N_{\text{-}x}, \hat{S}, O)$}
\IF{${cutType}={Improved}$ and $\exists i \in N \setminus \hat{S}$ such that $c(\hat{S}\cup \{i\})\leq Q$ }
\STATE{$(\hat{S},O)\leftarrow Greedy(N, \hat{S}, O)$}
\ENDIF
\STATE{Return $\hat{S}$ and ordering $O=(i_1, \ldots, i_{|\hat{S}|})$}
\end{algorithmic}
\label{Alg:S2}
\end{algorithm}

\begin{algorithm}[h]
\caption{S3}
{\textbf{Input:} A fractional leader solution $x^*$}\\
{\textbf{Output:} A follower solution $\hat{S} \in \mathcal{S}$ and an ordering of its elements}
\begin{algorithmic}[1]
\STATE{$\hat{S}\leftarrow \emptyset$ , $O=()$}
\WHILE{$\exists i \in N \setminus \hat{S}$: $c(\hat{S}\cup \{i\})\leq Q$ and $\max_{i\in N\setminus \hat{S}: \,c(\hat{S}\cup \{i\})\leq Q} v_i(\hat{S}) \geq 0$ }
\STATE{$i^* \leftarrow \arg \max_{i\in N\setminus \hat{S}: \,c(\hat{S}\cup \{i\})\leq Q} v_i(\hat{S})$}
\STATE{$\hat{S}\leftarrow \hat{S} \cup \{i^*\}$, $O.add(i^*)$}
\ENDWHILE
\STATE{Return $\hat{S}$ and ordering $O=(i_1, \ldots, i_{|\hat{S}|})$}
\end{algorithmic}
\label{Alg:S3}
\end{algorithm}

\subsection{Separation of Lifted and Alternative Cuts}
\label{section:Sep_lift_alt}
In our implementation, the lifted and alternative cuts are obtained heuristically, after the basic/improved cut is generated. For lifted cuts, as a preprocessing step the \textit{dominating list} $D_i$, which contains the items that can replace $i$ according to Theorem \ref{theo:liftCut4}, is computed for each $i\in N$ using the problem specific implications of superiority described in Remark \ref{rem:liftSpecial}.
Then, given a follower solution $\hat{S}$, sets $A$ and $B$ are initialized as empty sets and determined incrementally as follows. The items in $\hat{S}$ are sorted in non-increasing order of $\rho_i(\emptyset)$ values. The first item $i\in \hat{S}$ is picked and the value of $\big(\rho_j(\hat{S}\cup B)-\rho_i(\hat{S}\cup \{j\}\setminus \{i\}) \big)(1-x^*_j)$ is checked for each $j\in D_i\setminus \hat{S}$. If the maximum of these values is positive, $i$ is added to $A$ and the relevant $j$ is added to $B$, and they are not considered in further evaluations. Once all $i\in\hat{S}$ are considered, the final cut is reached. 

For an alternative cut, $A$ and $B$ are initialized as empty sets and obtained incrementally as follows. Given a follower solution $\hat{S}$, the items in $\hat{S}$ are sorted in non-increasing order of $\rho_i(\emptyset)$ values. Item $i$ is picked according to this order and the value of $\rho_{j}(\hat{S}\cup B\setminus \{i\})(x^*_{i}-x^*_{j})$ is checked for each $j \in N\setminus\hat{S}$ such that $c^\ell_{i}\geq c^\ell_{j}$ for each $\ell=1,\ldots ,L$. If the largest one of these values is positive, $i$ is added to $A$, $j$ is added to $B$ and they are not considered for further evaluations. Once all $i\in\hat{S}$ are processed, the resulting sets $A$ and $B$ yield the final alternative cut. Notice that this procedure would not yield a new cut for an integer leader solution $x^*$, as the integer separation procedure leads to $x^*_{i}=0$ for $i \in \hat{S}$. For this reason, alternative cuts are only generated for fractional $x^*$.



\section{Computational Results}
\label{section:Computational}
The algorithms we propose have been implemented in C++ using IBM ILOG CPLEX 12.10 as the MILP solver with its default settings. Each experiment uses a single thread of an Intel Xeon E5-2670v2 machine with 2.5 GHz processor. The time limit is 3600 seconds and the memory allocated to each experiment is 12 GB. We consider the two applications introduced in Section \ref{section:Applications} and generate random data sets of them to test our framework. In the following sections we present the instance generation procedures and the obtained results. All of the instances used are available at \url{https://msinnl.github.io/pages/bilevel.html}.

In our experiments, the following settings are considered for our B\&C:


\begin{itemize}
\item B: Only the basic cut \eqref{eq:basicCut} is used for separation.
\item I: Instead of a basic cut, an improved submodular interdiction cut \eqref{eq:ImprovedCut} is separated.
\item L: Once \eqref{eq:basicCut} or \eqref{eq:ImprovedCut} is obtained, it is lifted heuristically to \eqref{eq:liftCut4}.
\item D: Dominance inequalities are added to the initial model according to Theorem \ref{theo:dominanceValidity}.
\item A: In addition to the basic cut \eqref{eq:basicCut}, improved cut \eqref{eq:ImprovedCut}, or lifted cut \eqref{eq:liftCut4}, the alternative cut \eqref{eq:modCut2} is generated heuristically. 
\item E: For the separation of integer solutions, the enhanced procedure in Algorithm \ref{Alg:IntSepHeur} is used.
\end{itemize}
We include each of the components above incrementally. The basic setting is B-$s$ where $s\in\{\text{S1, S2, S3}\}$ denotes the method used to obtain $S$ for fractional $x^*$ and it is followed by I-$s$, IL-$s$, ILD-$s$, ILDA-$s$, and finally ILDAE-$s$ which includes all improvements and cut types we propose.

\subsection{Weighted Maximal Covering Interdiction Game}

WMCI instances used in our study are generated following a similar procedure proposed by \citet{revelle2008solving}. Customer coordinates are generated randomly in $[0,10]$. Potential facility locations are the same as the current customer locations, i.e., $n=m$, and $m \in \{50,60,70,80,90,100\}$. The profits $p_j$, $\forall j$ are randomly generated in $[1,100]$. Coverage is determined based on Euclidean distances and radius of coverage $r\in\{1,2,3\}$, i.e., a facility at location $i$ covers customer $j$ if $d_{ij}\leq r$ where $d_{ij}$ is the Euclidean distance between $i$ and $j$. The number of facilities to open is $B=0.1n$ and the interdiction budget $k$ takes value in $\{0.1 n, 0.2n\}$. Three instances are generated for each $(n,r,k)$ combination. 

Recall that condition $(ii)$ of Theorem \ref{theo:liftCut4} is equivalent to $\rho_{a_k}(b_k)=0$ for \WMCIG as explained in Remark \ref{rem:liftSpecial}. Therefore, the lifted cuts generated for this problem are in the form of \eqref{eq:liftCut2}.

The	plots of the results in terms of running times and final optimality gaps for each separation option $s\in\{S1,S2,S3\}$ are provided in Figures \ref{fig:WMCI S1}, \ref{fig:WMCI S2} and \ref{fig:WMCI S3}, respectively. The optimality gaps are obtained by $100\times (z^*-\underline{z})/(0.1 + z^*)$ where $z^*$ and $\underline{z}$ denote the objective value of the best integer solution and the best bound, respectively. We see in Figure \ref{fig:WMCI S1} that using improved cuts (I) instead of the basic one (B) causes a significant improvement in terms of running time and final optimality gaps. 
While the ratio of instances solved to optimality is 56\% under B-S1, it is increased to 73\% under I-S1. Adding lifted cuts (L) also improves both measures, especially final optimality gaps at the end of the time limit. 
The next component, dominance inequalities yields a significant improvement and the ratio of instances solved to optimality becomes 93\%. While the addition of alternative cuts to the improved/lifted ones does not make an apparent contribution to the performance, enhanced integer separation decreases the average solution time. The reason of the ineffectiveness of alternative cuts can be explained by the ground set definition used for S1, i.e, $x_i^*=0$ for $i\in \hat{S}$ except the cases in which some $i$ with $x_i^*>0$ are also included to reach a maximal set. This definition usually causes to have non-positive $(x_{a_k}-x_{b_k})$ values in the last term of alternative cuts which results in a smaller violation then the original improved cut. 

\begin{figure}[h!tb]
{\includegraphics[width=0.5\textwidth]{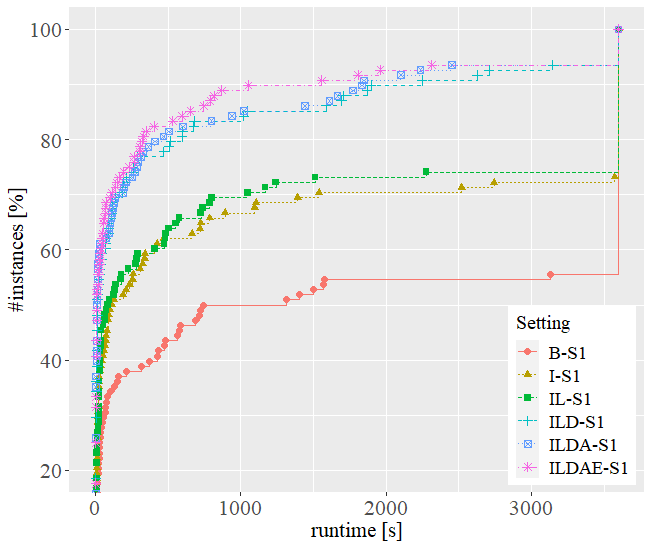}%
\includegraphics[width=0.5\textwidth]{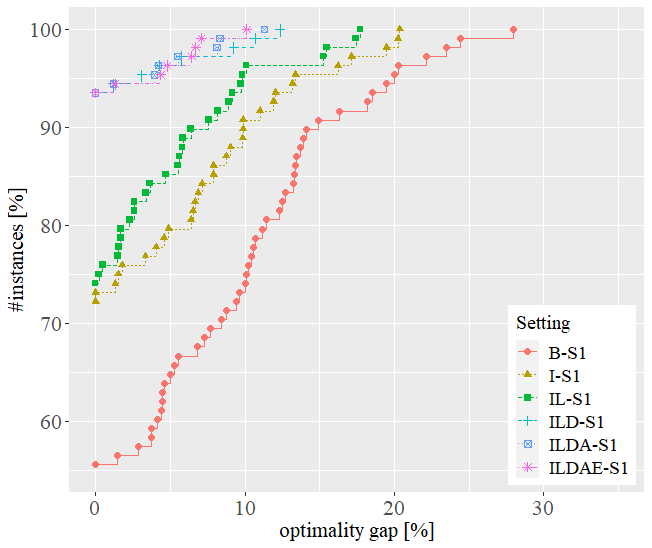}}
{\caption{WMCI S1 Results.}\label{fig:WMCI S1}}
\end{figure}

\begin{figure}[h!tb]
{\includegraphics[width=0.5\textwidth]{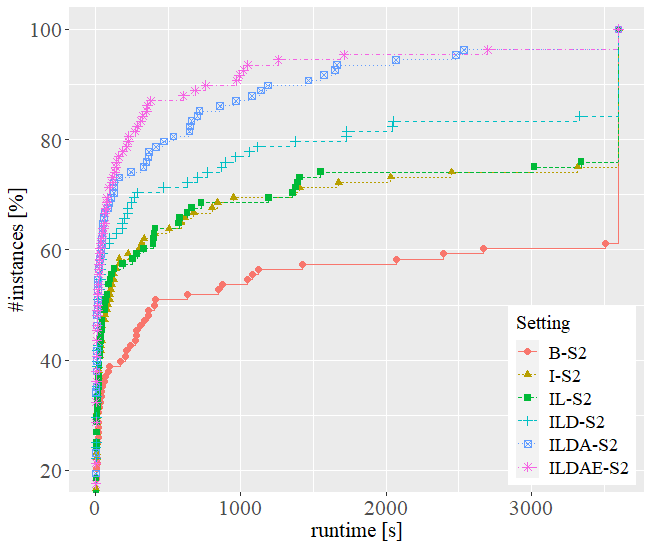}%
\includegraphics[width=0.5\textwidth]{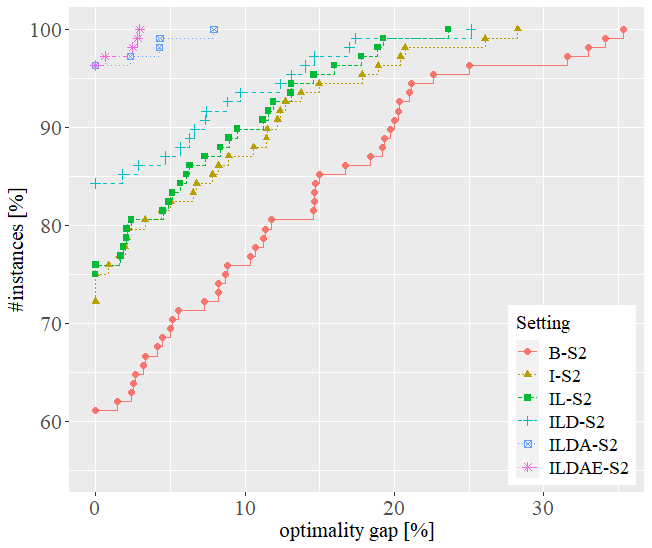}}
{\caption{WMCI S2 Results.}\label{fig:WMCI S2}}
\end{figure}

\begin{figure}[h!tb]
{\includegraphics[width=0.5\textwidth]{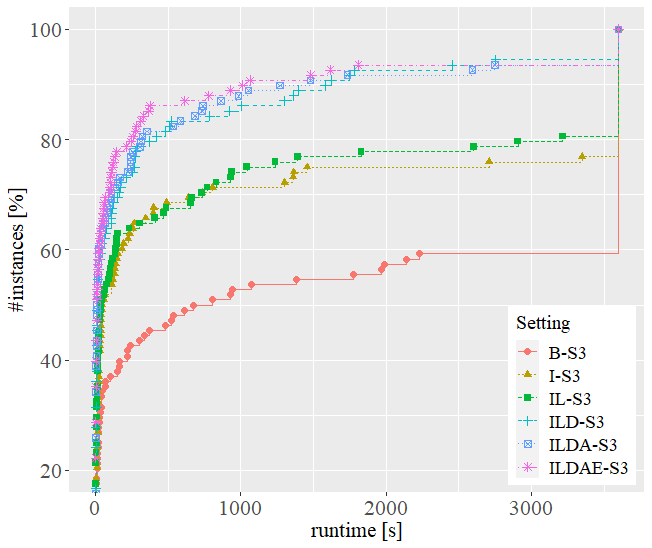}%
\includegraphics[width=0.5\textwidth]{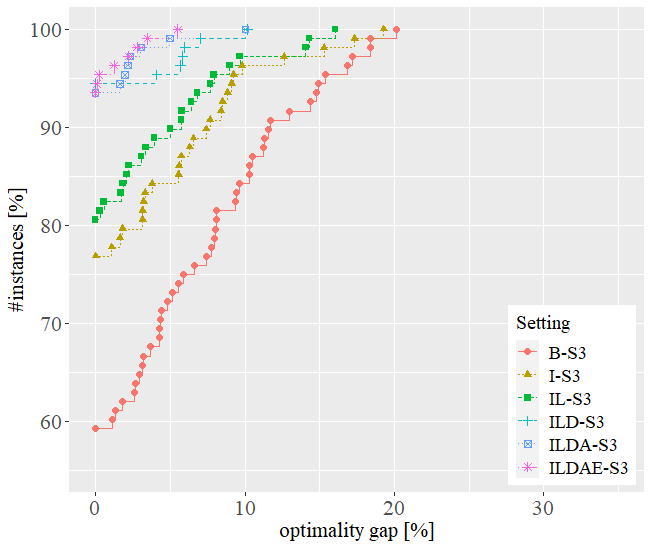}}
{\caption{WMCI S3 Results.}\label{fig:WMCI S3}}
\end{figure}

In Figure \ref{fig:WMCI S2}, we present the results for S2. Here, while 61\% of the instances are solved to optimality under the basic setting B-S2, the maximum optimality gap is 36\% which is large compared to B-S1. 
This value remains larger until the alternative cuts are included (setting ILDA-S2) which causes a substantial decrease in running time and final gap unlike option S1. In S2, the ground set for the follower problem is defined based on a rounding scheme. Thus, the enhanced separation procedure for alternative cuts given in Section \ref{section:Sep_lift_alt} is able to find eligible item pairs with $(x_{a_k}-x_{b_k})>0$ more easily, which explains the difference in the effect of component A under S1 and S2.
After the addition of the enhanced integer separation component (E), the solution times decrease more and the maximum optimality gap is reduced to 3\%. This result shows that a better search can be done when the time due to solving separation problems to optimality is saved.

The last option S3 whose results are plotted in Figure \ref{fig:WMCI S3}, yields a maximum gap of 20\% under the basic setting B-S3, which is notably smaller compared to S1 and S2, although the optimal solution ratio is similar to those of the previous options. On the other hand, I-S3, IL-S3, and ILD-S3 yield better running time and final gaps than their counterparts under S1 and S2. Since alternative cuts yield a slight performance improvement compared to S2, S3 falls barely behind of S2 in the complete setting ILDAE-S3, with a maximum gap of 5\%.

In Table \ref{table:WMCI}, the results of the complete (ILDAE) settings of all three fractional separation options are presented in terms of running time in seconds, final gaps, root gaps, the number of branch-and-cut tree nodes, and number of SICs generated. The first three measures are also compared to those obtained with the state-of-the-art MIBLP solver, using its default setting MIX{\tiny ++} \citep{fischetti2017new}. The MIBLP formulation of \WMCIG is provided in Appendix \ref{section:Appendix MCI model}, and the MIBLP solver is publicly available at \url{https://msinnl.github.io/pages/bilevel.html}. The numbers in the table show averages over three instances with the same parameter setting. We see that, while MIX{\tiny ++} is not able to solve any of the instances within the time limit of one hour and yields an average gap of $93.5\%$, this value is below $1\%$ with our settings. Even the minimum final gap obtained with MIX{\tiny ++}, which is not reported in the table, is 54\%. The difference between root gaps is also notable. The average root gap is almost 100\% with MIX{\tiny ++} as opposed to 25\% which is the average under S1, S2, and S3. When we focus only on our settings, we see that ILDAE-S2 is the best performing setting in terms of solution time, while it yields slightly larger root gaps than the others. The average tree size is considerably smaller under S3, and S2 requires the smallest number of cuts. The detailed results of all instances are given in Appendix \ref{section:Appendix results}, Tables \ref{table:Appendix-MCI-1}, \ref{table:Appendix-MCI-2}, and \ref{table:Appendix-MCI-3}.

\begin{landscape}
\begin{table}[h]
\caption{Results of \WMCIG experiments with the complete settings (ILDAE) of each separation option and with the benchmark MIBLP solver MIX\tiny++.}
\label{table:WMCI}
\fontsize{8pt}{11pt}\selectfont
{\centering
\begin{tabular}{l|rrrc|rrrr|rrrr|rrr|rrr}
  \midrule 
  & \multicolumn{4}{c|}{Time(sec.)} & \multicolumn{4}{c|}{Gap(\%)} & \multicolumn{4}{c|}{rGap(\%)} & \multicolumn{3}{c|}{\#Nodes} & \multicolumn{3}{c}{\#SIC}\\ $(n,B,k,r)$ & S1 & S2  &S3 &MIX\tiny++ & S1 & S2 &S3 &MIX\tiny++ & S1 & S2 &S3 &MIX\tiny++ & S1 & S2 &S3 & S1 & S2 &S3 \\ \midrule
    (50,5,5,1) & 0.5   & 0.1   & 2.8   & TL    & 0.0   & 0.0   & 0.0   & 66.7  & 23.3  & 14.3  & 17.3  & 99.7  & 91.3  & 53.0  & 50.7  & 161.7 & 173.7 & 208.3 \\
    (50,5,5,2) & 1.1   & 1.8   & 3.0   & TL    & 0.0   & 0.0   & 0.0   & 73.9  & 18.7  & 15.9  & 17.6  & 100.0 & 230.3 & 216.3 & 190.0 & 219.3 & 249.7 & 331.3 \\
    (50,5,5,3) & 4.0   & 3.2   & 3.0   & TL    & 0.0   & 0.0   & 0.0   & 84.3  & 14.3  & 15.0  & 15.6  & 99.6  & 309.7 & 376.0 & 311.0 & 151.3 & 174.3 & 188.7 \\
    (50,5,10,1) & 4.3   & 3.1   & 0.9   & TL    & 0.0   & 0.0   & 0.0   & 95.4  & 36.2  & 28.9  & 33.8  & 100.0 & 520.7 & 194.7 & 210.0 & 946.0 & 818.3 & 1148.7 \\
    (50,5,10,2) & 4.6   & 3.1   & 3.9   & TL    & 0.0   & 0.0   & 0.0   & 88.9  & 28.8  & 26.2  & 29.7  & 100.0 & 1506.3 & 730.0 & 933.0 & 1252.3 & 1149.3 & 1705.3 \\
    (50,5,10,3) & 4.7   & 1.9   & 2.5   & TL    & 0.0   & 0.0   & 0.0   & 80.1  & 22.9  & 27.1  & 23.4  & 100.0 & 266.0 & 351.0 & 203.3 & 293.3 & 361.7 & 482.0 \\
    (60,6,6,1) & 0.9   & 2.8   & 0.5   & TL    & 0.0   & 0.0   & 0.0   & 95.9  & 22.7  & 18.7  & 20.0  & 100.0 & 271.3 & 128.3 & 156.7 & 392.7 & 400.7 & 442.3 \\
    (60,6,6,2) & 5.8   & 32.8  & 12.7  & TL    & 0.0   & 0.0   & 0.0   & 89.1  & 19.6  & 18.7  & 21.5  & 100.0 & 1029.7 & 1254.7 & 953.7 & 645.0 & 785.3 & 869.7 \\
    (60,6,6,3) & 1.3   & 1.0   & 1.1   & TL    & 0.0   & 0.0   & 0.0   & 85.1  & 14.8  & 20.4  & 17.6  & 99.4  & 660.7 & 1101.7 & 1058.3 & 305.3 & 410.7 & 441.0 \\
    (60,6,12,1) & 1.6   & 1.8   & 1.9   & TL    & 0.0   & 0.0   & 0.0   & 97.0  & 33.7  & 29.1  & 25.4  & 100.0 & 553.7 & 168.0 & 188.3 & 1217.0 & 999.3 & 1121.0 \\
    (60,6,12,2) & 35.9  & 112.7 & 61.8  & TL    & 0.0   & 0.0   & 0.0   & 94.7  & 28.6  & 30.3  & 30.5  & 100.0 & 6615.7 & 6319.3 & 5817.0 & 3449.3 & 4701.7 & 6052.3 \\
    (60,6,12,3) & 4.6   & 8.3   & 4.0   & TL    & 0.0   & 0.0   & 0.0   & 86.8  & 23.6  & 30.0  & 25.9  & 100.0 & 1128.0 & 2254.0 & 957.7 & 874.7 & 1159.0 & 1540.0 \\
    (70,7,7,1) & 2.2   & 2.3   & 1.7   & TL    & 0.0   & 0.0   & 0.0   & 98.7  & 23.2  & 20.5  & 22.4  & 100.0 & 746.0 & 242.0 & 315.0 & 895.3 & 821.3 & 1043.0 \\
    (70,7,7,2) & 144.7 & 145.4 & 113.1 & TL    & 0.0   & 0.0   & 0.0   & 95.9  & 18.4  & 20.9  & 19.3  & 100.0 & 6815.0 & 10818.7 & 8437.3 & 1162.7 & 1131.7 & 1273.7 \\
    (70,7,7,3) & 1.1   & 2.8   & 1.6   & TL    & 0.0   & 0.0   & 0.0   & 85.6  & 15.3  & 21.0  & 17.6  & 100.0 & 911.3 & 2900.3 & 1828.3 & 425.7 & 596.0 & 495.3 \\
    (70,7,14,1) & 120.3 & 99.4  & 114.4 & TL    & 0.0   & 0.0   & 0.0   & 97.0  & 35.7  & 32.6  & 33.0  & 100.0 & 5688.3 & 1699.0 & 2403.3 & 9265.0 & 7121.0 & 9387.0 \\
    (70,7,14,2) & 820.5 & 1271.1 & 1323.3 & TL    & 0.0   & 0.2   & 0.4   & 99.6  & 28.2  & 32.0  & 30.6  & 100.0 & 41541.7 & 30312.7 & 24504.0 & 12161.0 & 10736.3 & 15653.7 \\
    (70,7,14,3) & 21.1  & 15.5  & 7.6   & TL    & 0.0   & 0.0   & 0.0   & 88.1  & 24.3  & 30.4  & 27.5  & 100.0 & 3059.0 & 2838.3 & 774.7 & 1605.0 & 1752.3 & 1928.0 \\
    (80,8,8,1) & 6.5   & 7.8   & 5.2   & TL    & 0.0   & 0.0   & 0.0   & 99.4  & 21.8  & 21.3  & 21.9  & 100.0 & 1067.0 & 354.0 & 636.7 & 1521.7 & 1110.7 & 1846.3 \\
    (80,8,8,2) & 189.6 & 152.4 & 154.2 & TL    & 0.0   & 0.0   & 0.0   & 99.8  & 17.9  & 19.8  & 20.6  & 100.0 & 13628.7 & 20936.3 & 11046.7 & 2165.7 & 2428.7 & 2662.0 \\
    (80,8,8,3) & 2.4   & 4.0   & 4.8   & TL    & 0.0   & 0.0   & 0.0   & 99.9  & 16.1  & 24.9  & 19.8  & 100.0 & 995.7 & 3898.3 & 2453.0 & 470.7 & 692.7 & 667.7 \\
    (80,8,16,1) & 186.4 & 45.9  & 84.7  & TL    & 0.0   & 0.0   & 0.0   & 99.7  & 34.2  & 36.7  & 35.3  & 100.0 & 9185.3 & 1261.7 & 2677.7 & 12947.3 & 6200.7 & 10748.3 \\
    (80,8,16,2) & 2116.9 & 1668.5 & 2939.0 & TL    & 1.6   & 1.0   & 1.0   & 99.6  & 27.7  & 30.1  & 29.0  & 100.0 & 56039.7 & 28043.3 & 32173.7 & 19053.7 & 17227.0 & 27288.7 \\
    (80,8,16,3) & 5.4   & 10.3  & 4.0   & TL    & 0.0   & 0.0   & 0.0   & 100.0 & 26.2  & 31.3  & 29.8  & 100.0 & 1213.0 & 2961.0 & 694.3 & 1090.3 & 1577.0 & 1229.0 \\
    (90,9,9,1) & 47.7  & 32.9  & 29.1  & TL    & 0.0   & 0.0   & 0.0   & 97.4  & 19.9  & 21.9  & 21.3  & 100.0 & 4124.3 & 867.3 & 1414.3 & 3910.0 & 3220.7 & 3393.3 \\
    (90,9,9,2) & 641.0 & 1110.9 & 1301.0 & TL    & 0.0   & 0.0   & 0.0   & 95.1  & 17.6  & 18.6  & 18.6  & 100.0 & 42742.0 & 61352.0 & 33670.0 & 4936.3 & 4836.0 & 5721.0 \\
    (90,9,9,3) & 5.1   & 6.8   & 2.3   & TL    & 0.0   & 0.0   & 0.0   & 83.0  & 18.2  & 23.8  & 20.5  & 100.0 & 373.7 & 5118.7 & 914.3 & 651.7 & 908.7 & 901.0 \\
    (90,9,18,1) & 1369.6 & 153.7 & 529.6 & TL    & 0.0   & 0.0   & 0.0   & 99.3  & 35.3  & 34.8  & 35.3  & 100.0 & 28151.0 & 2666.0 & 5221.7 & 30952.3 & 10967.0 & 23276.3 \\
    (90,9,18,2) & 2750.8 & 1555.2 & 1892.3 & TL    & 2.8   & 0.9   & 1.2   & 99.6  & 25.5  & 30.1  & 30.0  & 100.0 & 50290.7 & 26917.0 & 16770.3 & 21869.7 & 16837.3 & 24431.7 \\
    (90,9,18,3) & 24.2  & 45.1  & 19.6  & TL    & 0.0   & 0.0   & 0.0   & 99.9  & 26.4  & 32.0  & 29.4  & 100.0 & 1668.0 & 9765.7 & 1300.7 & 1485.3 & 2788.0 & 2380.0 \\
    (100,10,10,1) & 174.3 & 37.8  & 49.6  & TL    & 0.0   & 0.0   & 0.0   & 98.7  & 22.6  & 23.5  & 23.2  & 100.0 & 8819.3 & 1452.0 & 2357.3 & 9102.3 & 3957.0 & 5558.3 \\
    (100,10,10,2) & 393.1 & 399.6 & 248.8 & TL    & 0.0   & 0.0   & 0.0   & 97.3  & 17.4  & 21.0  & 19.8  & 100.0 & 21016.7 & 33557.3 & 16961.0 & 4808.7 & 4627.7 & 6257.3 \\
    (100,10,10,3) & 9.8   & 22.5  & 12.5  & TL    & 0.0   & 0.0   & 0.0   & 94.5  & 18.3  & 26.4  & 21.9  & 100.0 & 909.0 & 10693.7 & 3846.3 & 1018.3 & 1462.3 & 1345.3 \\
    (100,10,20,1) & TL    & 2519.7 & TL    & TL    & 7.8   & 0.8   & 2.6   & 99.6  & 34.3  & 35.7  & 37.2  & 100.0 & 18500.7 & 7836.0 & 10143.7 & 54296.7 & 37962.0 & 51913.7 \\
    (100,10,20,2) & 1648.9 & 499.2 & 389.8 & TL    & 1.4   & 0.0   & 0.0   & 100.0 & 26.9  & 31.5  & 29.6  & 100.0 & 28282.7 & 14168.3 & 3974.3 & 12621.7 & 6238.0 & 7334.3 \\
    (100,10,20,3) & 131.9 & 467.0 & 68.2  & TL    & 0.0   & 0.0   & 0.0   & 99.6  & 29.1  & 35.4  & 31.5  & 99.8  & 8242.7 & 34327.0 & 3764.7 & 3279.7 & 5546.0 & 4851.0 \\
    \midrule
    Average & 402.3 & 290.2 & 361.0 & TL    & 0.4   & 0.1   & 0.1   & 93.5  & 24.1  & 25.9  & 25.1  & 100.0 & 10199.9 & 9114.8 & 5536.5 & 6155.7 & 4503.6 & 6281.0 \\
    \midrule
\end{tabular}}
{\\ \footnotesize The results are aggregated over the three instances with the same $n$, $B$, $k$, and $r$ values, and given as averages. TL indicates that the time limit of 3600 seconds is reached for all instances involved in the average.}
\end{table}
\end{landscape}

\subsection{Bipartite Inference Interdiction Game}

While generating the \BIIG instances, we adopt the parameter settings used in \cite{salvagnin2019some} for the bipartite inference problem which constitutes the lower level of \BIIG. We do not include the parameter values that lead to failing to solve the problem within one hour according to their results, as we have an additional problem layer. As a result, the instances are generated as follows. The activating probability $p_i$ is sampled uniformly in $[0,1]$ for each $i \in N$. For the density $d$ of the graphs, i.e., the probability of having an arc between each $(i,j)$ pair, in addition to 0.07 which is the only value used in \cite{salvagnin2019some}, two more values $\{0.1,0.15\}$ are determined, and the arcs are generated in a completely random manner. The number of items $n\in\{20,50,100\}$, the number of targets $m\in \{2n,5n,10n\}$, and number of items $B$ that the follower can choose is in $\{5,10\}$ for $n=20$, $\{10,20\}$ for $n=50$, and equal to 10 for $n=100$. The leader can interdict $k=5$ items if $n=20$ and 10 items if $n>20$. Five distinct instances are generated for each parameter setting. 

The results of the experiments in terms of running time and final gaps are plotted in Figures \ref{fig:BII S1}, \ref{fig:BII S2} and \ref{fig:BII S3}. Common to all three separation options, the improved cuts (I) make the largest contribution for both measures. I, IL, and ILD settings perform very similarly. In the detailed results that are not reported here, we see that the number of lifted cuts and dominance inequalities added are very small which leads to different branch-and-cut trees but not notable improvement in performance. We attribute this situation to the rareness of item pairs suitable to be used in these cuts, due to the structure of the instances, i.e., it is difficult to have that one item covers all the targets that another item covers and has a larger activation probability as described in Remark \ref{rem:liftSpecial} and Proposition \ref{prop:dominanceBipartite}.
Under S1 (see Figure \ref{fig:BII S1}), the alternative cuts cause slightly smaller solution times and maximum gap reduced from 40\% to 30\%. Including component H further decreases this number to 20\%.

\begin{figure}[h!tb]
{\includegraphics[width=0.5\textwidth]{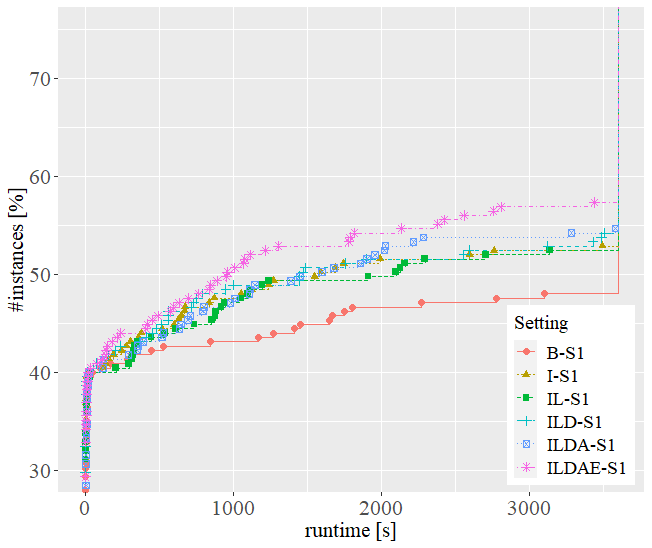}%
\includegraphics[width=0.5\textwidth]{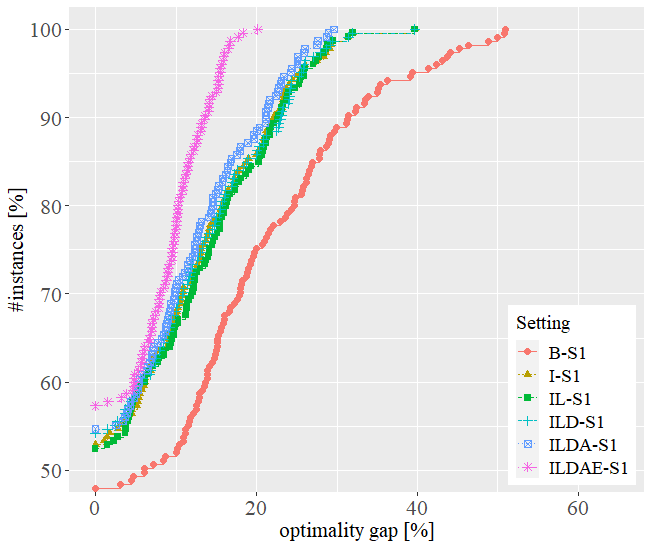}}
{\caption{\BIIG S1 Results.}\label{fig:BII S1}}
\end{figure}

\begin{figure}[h!tb]
{\includegraphics[width=0.5\textwidth]{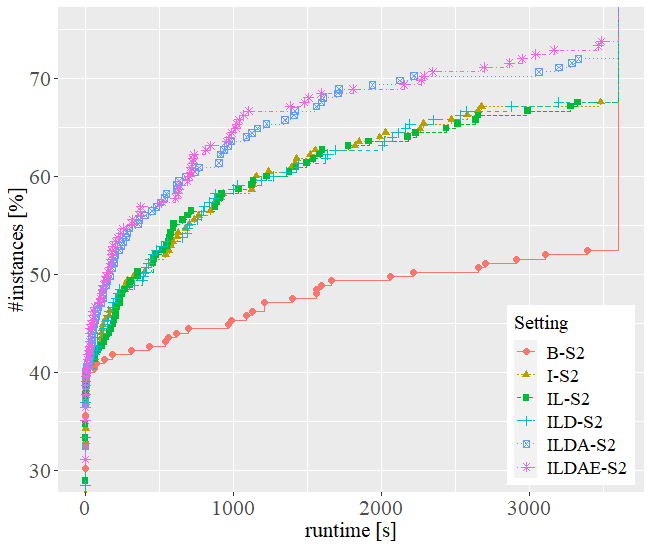}%
\includegraphics[width=0.5\textwidth]{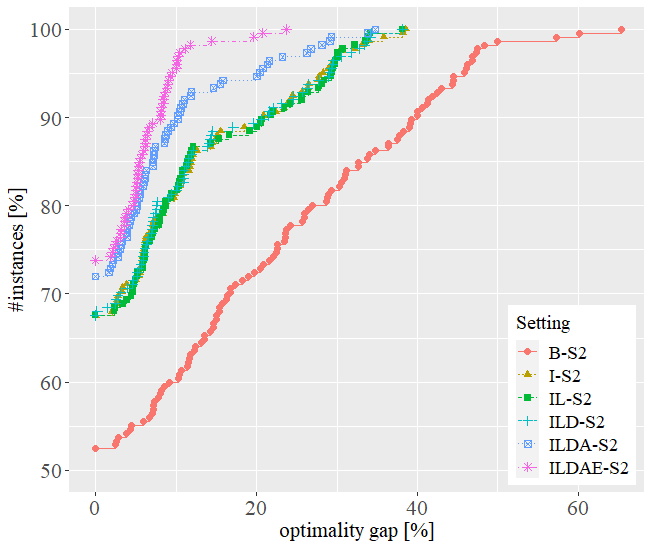}}
{\caption{\BIIG S2 Results.}\label{fig:BII S2}}
\end{figure}

\begin{figure}[h!tb]
{\includegraphics[width=0.5\textwidth]{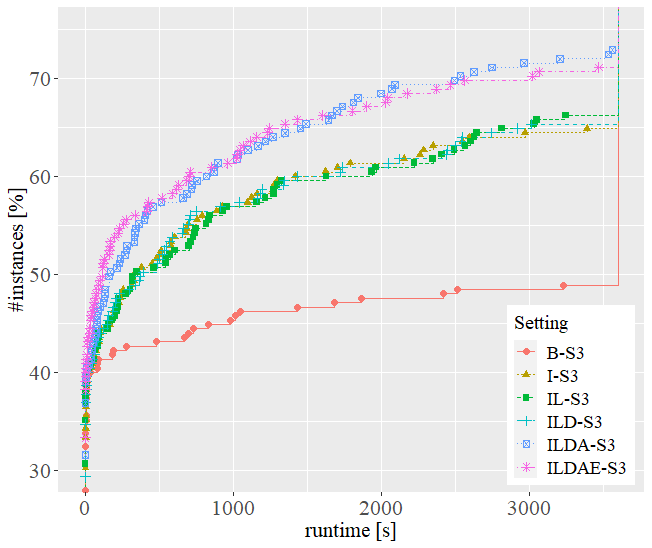}%
\includegraphics[width=0.5\textwidth]{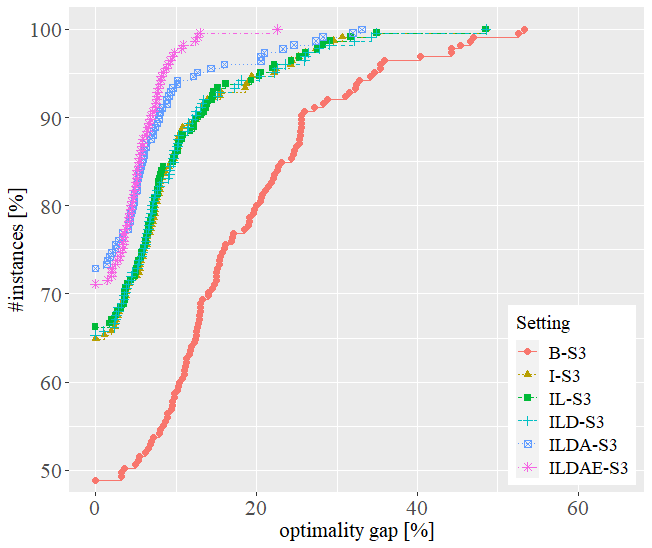}}
{\caption{\BIIG S3 Results.}\label{fig:BII S3}}
\end{figure}

The plots for S2 are shown in Figures \ref{fig:BII S2}. As is the case with \WMCIG instances, B-S1 setting yields very large final optimality gaps. The increase in the number of instances solved to optimality due to the improved SICs is larger compared to S1, as can be seen from both plots. Alternative cuts have a larger contribution than they have under S1. With all the components (ILDAE-S2), the maximum optimality gap is 24\% and the 74\% of instances are solved to optimality within the timelimit. 
The performance of S3 is similar to that of S1, except the maximum gap under the basic setting which is better in the former. 

Next, in Table \ref{table:BII} we present the average results for the ILDAE setting under all three options, with a similar structure as used in Table \ref{table:WMCI}, except the MIX{\tiny++} columns, since \BIIG does not fit into the MIX{\tiny++} setting, due to not having a compact MIBLP formulation. Each number denotes the average over five instances. In terms of running time and final gaps, while S2 and S3 perform similarly, S1 falls behind. S3 outperforms the others in terms of root gaps. As before, the tree size is smallest under S3. Since the number of instances solved to optimality is larger with this setting, it is understood that with S3 the optimal solution is reached in a smaller number of subproblems. Finally, the number of SICs generated is smaller when using S2. The detailed results for two instances from each class shown in Table \ref{table:BII} are provided in Appendix \ref{section:Appendix results}, Tables \ref{table:Appendix-BII-1}, \ref{table:Appendix-BII-2}, and \ref{table:Appendix-BII-3}.

\begin{adjustwidth}{+1cm}{-1cm}

\begin{table}[h!]
\caption{Results of \BIIG experiments with the complete settings (ILDAE) of each separation option. }
\label{table:BII}
{\fontsize{8pt}{10pt}\selectfont
\centering
\begin{tabular}{l|rrr|rrr|rrr|rrr|rrr}
  \midrule
  & \multicolumn{3}{c|}{Time(sec.)} & \multicolumn{3}{c|}{Gap(\%)} & \multicolumn{3}{c|}{rGap(\%)} & \multicolumn{3}{c|}{\#Nodes} & \multicolumn{3}{c}{\#SIC}\\ 
 $(n,m,B,k,d)$ & S1 & S2  &S3 & S1 & S2 &S3 & S1 & S2 &S3 & S1 & S2 &S3 & S1 & S2 &S3 \\ \midrule
    (20,40,5,5,0.07) & 0.8   & 0.1   & 0.1   & 0.0   & 0.0   & 0.0   & 40.8  & 25.8  & 25.0  & 151.6 & 63.4  & 71.0  & 185.0 & 208.0 & 253.4 \\
    (20,40,5,5,0.10) & 0.1   & 0.1   & 0.2   & 0.0   & 0.0   & 0.0   & 38.2  & 22.9  & 23.1  & 92.8  & 49.8  & 42.8  & 142.0 & 151.8 & 177.2 \\
    (20,40,5,5,0.15) & 1.2   & 0.3   & 0.5   & 0.0   & 0.0   & 0.0   & 37.7  & 27.0  & 24.1  & 139.2 & 56.6  & 50.0  & 194.4 & 202.8 & 252.2 \\
    (20,40,10,5,0.07) & 0.6   & 0.3   & 0.1   & 0.0   & 0.0   & 0.0   & 26.4  & 25.0  & 10.2  & 101.2 & 38.2  & 12.8  & 108.0 & 135.4 & 77.8 \\
    (20,40,10,5,0.10) & 2.1   & 0.1   & 0.4   & 0.0   & 0.0   & 0.0   & 27.4  & 26.2  & 16.5  & 174.0 & 63.6  & 28.4  & 182.2 & 194.0 & 135.4 \\
    (20,40,10,5,0.15) & 1.2   & 0.1   & 0.2   & 0.0   & 0.0   & 0.0   & 24.2  & 25.5  & 21.4  & 302.2 & 74.4  & 46.8  & 378.0 & 252.0 & 252.8 \\
    (20,100,5,5,0.07) & 1.8   & 0.2   & 1.0   & 0.0   & 0.0   & 0.0   & 37.5  & 23.0  & 23.0  & 93.4  & 32.8  & 43.8  & 152.6 & 144.2 & 185.2 \\
    (20,100,5,5,0.10) & 2.0   & 1.7   & 0.2   & 0.0   & 0.0   & 0.0   & 39.9  & 27.7  & 25.0  & 144.0 & 57.6  & 63.2  & 192.6 & 221.2 & 253.8 \\
    (20,100,5,5,0.15) & 3.2   & 0.6   & 0.2   & 0.0   & 0.0   & 0.0   & 34.8  & 24.4  & 25.8  & 119.4 & 39.4  & 45.8  & 202.0 & 203.8 & 275.8 \\
    (20,100,10,5,0.07) & 3.3   & 3.1   & 1.5   & 0.0   & 0.0   & 0.0   & 27.3  & 26.6  & 16.8  & 278.0 & 78.0  & 25.6  & 289.8 & 283.2 & 145.4 \\
    (20,100,10,5,0.10) & 10.2  & 1.7   & 0.7   & 0.0   & 0.0   & 0.0   & 28.1  & 27.3  & 20.7  & 620.4 & 168.0 & 93.6  & 667.0 & 589.4 & 500.0 \\
    (20,100,10,5,0.15) & 8.8   & 0.7   & 0.3   & 0.0   & 0.0   & 0.0   & 28.6  & 29.4  & 24.3  & 497.8 & 140.2 & 71.0  & 582.2 & 471.4 & 432.4 \\
    (20,200,5,5,0.07) & 0.2   & 0.2   & 0.2   & 0.0   & 0.0   & 0.0   & 38.6  & 23.2  & 23.2  & 116.4 & 36.6  & 45.0  & 167.4 & 143.0 & 172.4 \\
    (20,200,5,5,0.10) & 2.4   & 1.0   & 0.2   & 0.0   & 0.0   & 0.0   & 39.8  & 24.1  & 22.9  & 122.2 & 57.4  & 61.6  & 200.6 & 221.8 & 272.4 \\
    (20,200,5,5,0.15) & 0.6   & 0.4   & 1.1   & 0.0   & 0.0   & 0.0   & 40.4  & 27.6  & 27.9  & 174.4 & 55.6  & 67.6  & 287.0 & 279.6 & 340.4 \\
    (20,200,10,5,0.07) & 1.0   & 0.2   & 0.4   & 0.0   & 0.0   & 0.0   & 25.8  & 24.2  & 15.5  & 273.4 & 67.8  & 32.8  & 295.2 & 252.0 & 179.0 \\
    (20,200,10,5,0.10) & 6.0   & 1.8   & 0.8   & 0.0   & 0.0   & 0.0   & 29.6  & 29.0  & 20.8  & 642.6 & 169.4 & 116.0 & 852.6 & 615.2 & 586.4 \\
    (20,200,10,5,0.15) & 3.6   & 3.1   & 2.0   & 0.0   & 0.0   & 0.0   & 29.0  & 28.6  & 27.1  & 1242.4 & 361.2 & 206.6 & 1509.4 & 1062.8 & 1127.6 \\
    (50,100,10,10,0.07) & 742.6 & 71.1  & 54.5  & 0.0   & 0.0   & 0.0   & 39.8  & 33.6  & 33.7  & 13243.2 & 1066.4 & 1100.6 & 19297.2 & 5950.6 & 7381.0 \\
    (50,100,10,10,0.10) & 495.0 & 70.3  & 44.7  & 0.0   & 0.0   & 0.0   & 40.4  & 35.5  & 33.4  & 10887.0 & 1135.0 & 1079.6 & 19737.0 & 6451.4 & 7945.4 \\
    (50,100,10,10,0.15) & 1251.0 & 291.8 & 291.5 & 0.0   & 0.0   & 0.0   & 38.1  & 36.9  & 36.2  & 17118.2 & 3012.8 & 2671.4 & 30409.2 & 13284.6 & 17265.2 \\
    (50,100,20,10,0.07) & 3094.1 & 167.3 & 74.0  & 3.7   & 0.0   & 0.0   & 25.4  & 27.4  & 24.8  & 30110.6 & 2761.2 & 980.4 & 44845.6 & 14129.6 & 7556.6 \\
    (50,100,20,10,0.10) & TL    & 3426.5 & 3493.3 & 12.6  & 2.2   & 3.0   & 29.1  & 27.7  & 28.0  & 17823.4 & 13535.0 & 5847.4 & 43050.2 & 60934.2 & 50894.8 \\
    (50,100,20,10,0.15) & TL    & 3176.4 & 3129.2 & 11.9  & 6.2   & 5.1   & 28.2  & 29.2  & 28.2  & 13805.8 & 10767.4 & 5591.4 & 36628.6 & 46254.8 & 48118.4 \\
    (50,250,10,10,0.07) & 1254.7 & 214.3 & 195.4 & 1.1   & 0.0   & 0.0   & 38.9  & 34.1  & 34.3  & 17169.8 & 1778.8 & 1816.0 & 33312.2 & 11623.2 & 14706.6 \\
    (50,250,10,10,0.10) & 1476.3 & 344.1 & 444.3 & 0.7   & 0.0   & 0.0   & 38.6  & 35.8  & 36.0  & 20454.2 & 3021.6 & 2607.0 & 38334.8 & 17439.0 & 22967.8 \\
    (50,250,10,10,0.15) & 2481.5 & 596.7 & 571.5 & 1.0   & 0.0   & 0.0   & 40.4  & 38.8  & 38.2  & 27197.2 & 4770.8 & 3521.4 & 52997.4 & 25344.8 & 31984.4 \\
    (50,250,20,10,0.07) & TL    & 2832.2 & 1663.8 & 10.5  & 3.0   & 1.5   & 29.3  & 28.7  & 27.9  & 24945.8 & 7078.4 & 3236.4 & 54231.0 & 44474.8 & 31120.0 \\
    (50,250,20,10,0.10) & TL    & TL    & TL    & 14.6  & 7.1   & 6.2   & 31.1  & 29.4  & 29.9  & 23904.0 & 7229.2 & 4432.6 & 55816.2 & 53267.2 & 53847.2 \\
    (50,250,20,10,0.15) & TL    & TL    & TL    & 13.2  & 7.2   & 7.8   & 29.4  & 29.3  & 28.3  & 17563.8 & 10147.0 & 4457.4 & 52104.4 & 54222.2 & 52020.0 \\
    (50,500,10,10,0.07) & 1394.1 & 239.9 & 244.2 & 1.2   & 0.0   & 0.0   & 41.2  & 36.3  & 34.7  & 18172.4 & 1909.6 & 2095.6 & 36983.0 & 12595.4 & 17501.2 \\
    (50,500,10,10,0.10) & 2286.9 & 1470.1 & 1452.6 & 4.8   & 0.7   & 0.7   & 40.2  & 36.8  & 37.4  & 18616.2 & 4365.2 & 3773.4 & 48691.2 & 30415.6 & 38090.8 \\
    (50,500,10,10,0.15) & 2457.4 & 616.0 & 1276.7 & 3.0   & 0.0   & 0.0   & 38.9  & 38.9  & 38.0  & 19249.4 & 4608.0 & 4069.2 & 46038.0 & 26010.8 & 37769.0 \\
    (50,500,20,10,0.07) & TL    & 2470.5 & 2577.9 & 11.9  & 3.1   & 3.6   & 28.7  & 28.3  & 27.6  & 21951.0 & 5431.0 & 3118.6 & 51392.4 & 41838.0 & 37029.0 \\
    (50,500,20,10,0.10) & TL    & TL    & TL    & 16.6  & 8.8   & 7.9   & 30.5  & 30.1  & 29.7  & 16875.2 & 6230.8 & 4408.8 & 47643.2 & 50604.0 & 57941.8 \\
    (50,500,20,10,0.15) & TL    & TL    & TL    & 14.9  & 11.7  & 12.2  & 30.1  & 29.8  & 30.1  & 14172.6 & 6352.2 & 3858.2 & 48670.0 & 41906.6 & 46339.2 \\
    (100,200,10,10,0.07) & 3095.6 & 906.7 & 1607.6 & 4.8   & 0.0   & 0.7   & 31.3  & 30.2  & 30.5  & 21226.4 & 6033.8 & 6315.4 & 58458.4 & 24640.2 & 38734.6 \\
    (100,200,10,10,0.10) & 3431.6 & 1880.5 & 1824.4 & 6.7   & 1.1   & 1.2   & 29.7  & 31.5  & 31.5  & 20834.4 & 8998.0 & 7423.8 & 53874.6 & 31265.0 & 43205.0 \\
    (100,200,10,10,0.15) & TL    & 2906.7 & TL    & 9.7   & 4.7   & 3.6   & 28.4  & 32.5  & 31.4  & 10258.2 & 8524.6 & 6958.6 & 36347.2 & 24304.2 & 49364.6 \\
    (100,500,10,10,0.07) & TL    & 2218.3 & 2543.1 & 8.2   & 1.1   & 1.7   & 31.9  & 31.6  & 31.5  & 15180.4 & 6239.4 & 5944.8 & 63577.6 & 41637.6 & 54433.2 \\
    (100,500,10,10,0.10) & TL    & 3349.6 & TL    & 9.6   & 3.8   & 4.3   & 30.3  & 32.3  & 31.8  & 14783.8 & 7862.6 & 6286.6 & 60239.8 & 49257.4 & 67210.4 \\
    (100,500,10,10,0.15) & TL    & 3577.4 & TL    & 12.7  & 9.5   & 6.2   & 29.0  & 34.2  & 32.4  & 6514.2 & 6679.8 & 4145.4 & 31428.2 & 26154.8 & 46963.4 \\
    (100,1000,10,10,0.07) & 3568.3 & 2936.3 & 3315.2 & 7.0   & 3.0   & 3.6   & 31.4  & 33.3  & 31.9  & 15108.0 & 4373.4 & 5407.8 & 68691.6 & 42208.6 & 63986.0 \\
    (100,1000,10,10,0.10) & TL    & 3052.2 & 3484.5 & 8.6   & 2.7   & 4.2   & 31.4  & 34.6  & 33.5  & 13707.4 & 6304.0 & 5070.0 & 66499.8 & 48547.6 & 62304.0 \\
    (100,1000,10,10,0.15) & TL    & TL    & TL    & 12.5  & 9.6   & 8.3   & 29.7  & 36.3  & 33.3  & 6278.2 & 6410.6 & 3571.6 & 34473.2 & 33760.2 & 43062.8 \\ \midrule
    Average & 1721.7 & 1218.5 & 1268.9 & 4.5   & 1.9   & 1.8   & 33.0  & 30.0  & 27.9  & 10498.6 & 3516.4 & 2464.8 & 27563.6 & 19647.9 & 23452.5 \\
     \midrule
\end{tabular} }
{\footnotesize The results are aggregated over the five instances with the same $n$, $m$, $B$, $k$, and $d$ values, and given as averages. TL indicates that the time limit of 3600 seconds is reached for all instances involved in the average.}
\end{table}

\end{adjustwidth}

\section{Conclusion}
\label{section:Conclusion}

In this paper, we have presented an exact method to solve interdiction games with a submodular and non-decreasing objective function. Such problems have many real world applications as described in Section \ref{section:ProblemDefinition}. We introduce submodular interdiction cuts (SIC) by exploiting the special properties of submodular set functions. We also develop improved and lifted variants of these SIC.
The branch-and-cut framework which we design based on SICs involves several other components such as dominance inequalities, greedy algorithms for separation of fractional solutions and an enhanced separation procedure for integer solutions. 
We also investigate the impact of using maximal sets while building SICs instead of non-maximal ones, and utilize the obtained information to design better separation schemes. 

To assess the performance of our solution algorithm and its individual components, we conduct a computational study on the weighted maximal covering interdiction game and the bipartite inference interdiction game. The results show that the components of our framework provide significant improvements with respect to the basic version. Moreover, our method vastly outperforms a state-of-the-art general purpose mixed-integer bilevel linear programming (MIBLP) solver for the weighted maximal covering interdiction game (for which a MIBLP formulation is possible). 

Regarding further work, a natural extension of interdiction games is the \emph{fortification problem} where a third problem layer includes the interdiction game as a constraint. There are several studies addressing defender-attacker-defender games such as \cite{cappanera2011optimal}, \cite{lozano2017backward}, \cite{lozano2017solving}, and \cite{zheng2018exact}. It could be interesting to study such games with submodular objective function.	Another possible future research direction could be developing methods for the solution of stochastic or robust submodular interdiction games. Finally, one could also focus on concrete submodular interdiction games and try to extend our general-purpose framework with problem-specific components.

\section*{Acknowledgments}
{
	The research was supported by the Linz Institute of Technology (Project LIT-2019-7-YOU-211) and the JKU Business School. LIT is funded by the state of Upper Austria.
}

\bibliographystyle{informs2014}
\bibliography{Submodular}

\clearpage
\appendix

\section{The MIBLP Formulation of \WMCIG }
\label{section:Appendix MCI model}


Let binary leader variables $x_i$, $i \in N$ indicate the interdiction decisions of the leader. Let binary follower variables $y_i$, $i \in N$ take the value one if and only if facility $i \in N$ is open in a solution. 
Let binary follower variables $z_j$, $j \in J$ take the value one if and only if customer $j$ is covered in a solution. The following formulation is used while solving \WMCIG instances by the MIBLP solver MIX{\tiny++}. 

\begin{align}
\min_{x\in X} \, \max_{y,z} &\, \sum_{j\in J}p_j z_j \notag \\
\text{s.t.} &\, \sum_{i\in N}y_i\leq B  \notag\\
& z_j \leq \sum_{i:j\in J(i)}y_i  &\forall j&\in J  \notag\\
& y_i\leq 1-x_i &\forall i& \in N  \notag \\
& y_i\in \{0,1\}&\forall i& \in N  \notag \\
& z_j\in \{0,1\}&\forall j& \in J \notag
\end{align}
where $X=\{x\in \{0,1\}^{|N|}: \sum_{i\in N}x_i \leq k \}$.

\section{Detailed Results}
\label{section:Appendix results}

\begin{landscape}
\begin{table}[h!]
\footnotesize \centering
\caption{Detailed \WMCIG results under separation option S2: basic and complete setting comparison.}
\label{table:Appendix-MCI-1}
\begin{tabular}{l|rrrrrrr|rrrrrrr}
\midrule
  & \multicolumn{7}{c|}{B-S2} & \multicolumn{7}{c}{ILDAE-S2} \\ \midrule 
  $(n,B,p,r,ID)$ & t(s.) & UB & LB & Gap(\%) & rGap(\%) & \#Nodes & \#SIC & t(s.)& UB & LB & Gap(\%) & rGap(\%) & \#Nodes & \#SIC\\ \midrule
(50,5,5,1,1) & 1 & 882 &  882.0 &  0.0 & 30.3 & 118 & 154 & 0 & 882 &  882.0 & 0.0 & 12.2 & 11 & 80 \\ 
  (50,5,5,1,2) & 1 & 883 &  883.0 &  0.0 & 31.4 & 258 & 240 & 0 & 883 &  883.0 & 0.0 & 13.1 & 78 & 213 \\ 
  (50,5,5,1,3) & 3 & 1003 & 1003.0 &  0.0 & 27.2 & 224 & 186 & 0 & 1003 & 1003.0 & 0.0 & 17.5 & 70 & 228 \\ 
  (50,5,5,2,1) & 12 & 1852 & 1852.0 &  0.0 & 24.1 & 658 & 178 & 1 & 1852 & 1852.0 & 0.0 & 16.5 & 58 & 200 \\ 
  (50,5,5,2,2) & 3 & 1835 & 1835.0 &  0.0 & 25.9 & 511 & 183 & 1 & 1835 & 1835.0 & 0.0 & 14.6 & 181 & 278 \\ 
  (50,5,5,2,3) & 12 & 2031 & 2031.0 &  0.0 & 28.9 & 730 & 184 & 4 & 2031 & 2031.0 & 0.0 & 16.6 & 410 & 271 \\ 
  (50,5,5,3,1) & 2 & 2442 & 2442.0 &  0.0 & 27.6 & 1085 & 233 & 1 & 2442 & 2442.0 & 0.0 & 15.2 & 306 & 153 \\ 
  (50,5,5,3,2) & 9 & 2171 & 2171.0 &  0.0 & 25.5 & 1593 & 164 & 8 & 2171 & 2171.0 & 0.0 & 14.1 & 435 & 170 \\ 
  (50,5,5,3,3) & 5 & 2383 & 2383.0 &  0.0 & 26.4 & 953 & 125 & 0 & 2383 & 2383.0 & 0.0 & 15.8 & 387 & 200 \\ 
  (50,5,10,1,1) & 8 & 828 &  827.9 &  0.0 & 41.7 & 2006 & 1432 & 0 & 828 &  828.0 & 0.0 & 25.5 & 158 & 594 \\ 
  (50,5,10,1,2) & 2 & 933 &  933.0 &  0.0 & 45.4 & 1777 & 1069 & 4 & 933 &  933.0 & 0.0 & 31.2 & 93 & 457 \\ 
  (50,5,10,1,3) & 8 & 720 &  720.0 &  0.0 & 43.8 & 3861 & 2168 & 5 & 720 &  720.0 & 0.0 & 29.9 & 333 & 1404 \\ 
  (50,5,10,2,1) & 5 & 1747 & 1746.9 &  0.0 & 40.0 & 4333 & 1247 & 1 & 1747 & 1747.0 & 0.0 & 26.4 & 396 & 960 \\ 
  (50,5,10,2,2) & 8 & 1587 & 1586.8 &  0.0 & 39.8 & 7927 & 1373 & 3 & 1587 & 1587.0 & 0.0 & 25.9 & 827 & 1239 \\ 
  (50,5,10,2,3) & 19 & 1634 & 1634.0 &  0.0 & 36.6 & 7230 & 1174 & 5 & 1634 & 1633.9 & 0.0 & 26.5 & 967 & 1249 \\ 
  (50,5,10,3,1) & 30 & 2353 & 2353.0 &  0.0 & 37.7 & 6482 & 942 & 3 & 2353 & 2353.0 & 0.0 & 27.0 & 243 & 270 \\ 
  (50,5,10,3,2) & 10 & 2287 & 2286.9 &  0.0 & 40.8 & 5371 & 911 & 1 & 2287 & 2287.0 & 0.0 & 28.6 & 257 & 285 \\ 
  (50,5,10,3,3) & 21 & 2024 & 2023.8 &  0.0 & 40.9 & 13688 & 1004 & 3 & 2024 & 2024.0 & 0.0 & 25.8 & 553 & 530 \\ 
  (60,6,6,1,1) & 1 & 1115 & 1115.0 &  0.0 & 30.0 & 571 & 513 & 8 & 1115 & 1115.0 & 0.0 & 19.0 & 83 & 262 \\ 
  (60,6,6,1,2) & 12 & 1391 & 1391.0 &  0.0 & 31.6 & 1108 & 529 & 0 & 1391 & 1391.0 & 0.0 & 17.8 & 99 & 383 \\ 
  (60,6,6,1,3) & 1 & 1326 & 1326.0 &  0.0 & 31.3 & 1064 & 506 & 1 & 1326 & 1326.0 & 0.0 & 19.3 & 203 & 557 \\ 
  (60,6,6,2,1) & 71 & 2625 & 2625.0 &  0.0 & 26.8 & 4338 & 429 & 66 & 2625 & 2625.0 & 0.0 & 18.6 & 1972 & 791 \\ 
  (60,6,6,2,2) & 24 & 2463 & 2463.0 &  0.0 & 26.3 & 3006 & 507 & 7 & 2463 & 2462.9 & 0.0 & 18.2 & 873 & 741 \\ 
  (60,6,6,2,3) & 62 & 2238 & 2238.0 &  0.0 & 32.0 & 1650 & 498 & 25 & 2238 & 2238.0 & 0.0 & 19.4 & 919 & 824 \\ 
  (60,6,6,3,1) & 3 & 2994 & 2994.0 &  0.0 & 29.5 & 6562 & 560 & 1 & 2994 & 2994.0 & 0.0 & 22.0 & 593 & 372 \\ 
  (60,6,6,3,2) & 4 & 2813 & 2812.8 &  0.0 & 32.0 & 3650 & 482 & 1 & 2813 & 2812.9 & 0.0 & 19.4 & 949 & 335 \\ 
  (60,6,6,3,3) & 8 & 2944 & 2944.0 &  0.0 & 29.0 & 8824 & 544 & 1 & 2944 & 2943.8 & 0.0 & 19.7 & 1763 & 525 \\ 
  (60,6,12,1,1) & 11 & 1031 & 1031.0 &  0.0 & 38.7 & 4281 & 2638 & 1 & 1031 & 1031.0 & 0.0 & 27.2 & 116 & 781 \\ 
  (60,6,12,1,2) & 37 & 935 &  935.0 &  0.0 & 42.9 & 12834 & 5639 & 3 & 935 &  935.0 & 0.0 & 31.1 & 341 & 1840 \\ 
  (60,6,12,1,3) & 7 & 995 &  995.0 &  0.0 & 43.0 & 2073 & 2060 & 2 & 995 &  995.0 & 0.0 & 29.2 & 47 & 377 \\ 
  (60,6,12,2,1) & 334 & 2385 & 2384.8 &  0.0 & 34.8 & 65219 & 5285 & 186 & 2385 & 2384.8 & 0.0 & 29.0 & 3819 & 4670 \\ 
  (60,6,12,2,2) & 364 & 2029 & 2028.8 &  0.0 & 36.9 & 102142 & 6493 & 140 & 2029 & 2028.8 & 0.0 & 31.9 & 13480 & 6947 \\ 
  (60,6,12,2,3) & 95 & 2252 & 2252.0 &  0.0 & 36.7 & 33502 & 5148 & 12 & 2252 & 2251.9 & 0.0 & 30.1 & 1659 & 2488 \\ 
  (60,6,12,3,1) & 214 & 2493 & 2492.8 &  0.0 & 46.9 & 64054 & 4787 & 1 & 2493 & 2493.0 & 0.0 & 31.9 & 304 & 467 \\ 
  (60,6,12,3,2) & 307 & 2673 & 2673.0 &  0.0 & 40.8 & 73903 & 5019 & 14 & 2673 & 2672.9 & 0.0 & 29.1 & 1989 & 1173 \\ 
  (60,6,12,3,3) & 283 & 2547 & 2547.0 &  0.0 & 47.8 & 116603 & 4638 & 10 & 2547 & 2546.8 & 0.0 & 29.0 & 4469 & 1837 \\ 
    \midrule
\end{tabular}
\end{table}
\end{landscape}

\begin{landscape}
\begin{table}
\caption{Detailed \WMCIG results under separation option S2: basic and complete setting comparison (continued).}
\label{table:Appendix-MCI-2}
\centering
\footnotesize
\begin{tabular}{l|rrrrrrr|rrrrrrr}
\midrule
  & \multicolumn{7}{c|}{B-S2} & \multicolumn{7}{c}{ILDAE-S2} \\ \midrule 
  $(n,B,p,r,ID)$ & t(s.) & UB & LB & Gap(\%) & rGap(\%) & \#Nodes & \#SIC & t(s.)& UB & LB & Gap(\%) & rGap(\%) & \#Nodes & \#SIC\\ \midrule
   (70,7,7,1,1) & 7 & 1441 & 1441.0 &  0.0 & 31.7 & 4292 & 1952 & 5 & 1441 & 1441.0 & 0.0 & 19.4 & 278 & 1017 \\ 
  (70,7,7,1,2) & 6 & 1427 & 1427.0 &  0.0 & 31.0 & 3745 & 1645 & 2 & 1427 & 1427.0 & 0.0 & 22.4 & 243 & 859 \\ 
  (70,7,7,1,3) & 6 & 1567 & 1567.0 &  0.0 & 30.4 & 2639 & 1910 & 1 & 1567 & 1567.0 & 0.0 & 19.7 & 205 & 588 \\ 
  (70,7,7,2,1) & 243 & 2785 & 2785.0 &  0.0 & 27.6 & 8306 & 1738 & 76 & 2785 & 2785.0 & 0.0 & 20.2 & 1200 & 1104 \\ 
  (70,7,7,2,2) & 366 & 3225 & 3224.7 &  0.0 & 26.2 & 22000 & 1262 & 131 & 3225 & 3224.7 & 0.0 & 20.2 & 9623 & 1324 \\ 
  (70,7,7,2,3) & 847 & 3122 & 3121.8 &  0.0 & 29.1 & 55435 & 796 & 229 & 3122 & 3121.8 & 0.0 & 22.2 & 21633 & 967 \\ 
  (70,7,7,3,1) & 45 & 3196 & 3195.7 &  0.0 & 28.3 & 42733 & 1789 & 5 & 3196 & 3195.7 & 0.0 & 22.5 & 6253 & 611 \\ 
  (70,7,7,3,2) & 9 & 3293 & 3292.8 &  0.0 & 33.4 & 11951 & 1558 & 2 & 3293 & 3292.7 & 0.0 & 20.3 & 1391 & 762 \\ 
  (70,7,7,3,3) & 42 & 3318 & 3317.7 &  0.0 & 34.5 & 31078 & 1285 & 2 & 3318 & 3318.0 & 0.0 & 20.2 & 1057 & 415 \\ 
  (70,7,14,1,1) & 86 & 1288 & 1287.9 &  0.0 & 40.4 & 13126 & 8263 & 3 & 1288 & 1288.0 & 0.0 & 31.3 & 321 & 1200 \\ 
  (70,7,14,1,2) & TL & 1400 & 1355.5 &  3.2 & 45.0 & 179438 & 50192 & 289 & 1400 & 1399.9 & 0.0 & 33.5 & 4358 & 18345 \\ 
  (70,7,14,1,3) & 206 & 1331 & 1330.9 &  0.0 & 42.6 & 18283 & 11931 & 6 & 1331 & 1331.0 & 0.0 & 33.2 & 418 & 1818 \\ 
  (70,7,14,2,1) & TL & 2940 & 2776.2 &  5.6 & 39.5 & 179608 & 20917 & 159 & 2940 & 2939.7 & 0.0 & 33.6 & 11097 & 7885 \\ 
  (70,7,14,2,2) & 873 & 2647 & 2646.7 &  0.0 & 38.5 & 48718 & 6268 & 55 & 2647 & 2647.0 & 0.0 & 30.0 & 596 & 1138 \\ 
  (70,7,14,2,3) & TL & 3001 & 2847.4 &  5.1 & 37.3 & 306010 & 22123 & TL & 3001 & 2982.2 & 0.6 & 32.4 & 79245 & 23186 \\ 
  (70,7,14,3,1) & TL & 3409 & 3317.9 &  2.7 & 40.5 & 221008 & 13647 & 13 & 3409 & 3409.0 & 0.0 & 31.3 & 4102 & 1830 \\ 
  (70,7,14,3,2) & 2396 & 3065 & 3064.7 &  0.0 & 40.0 & 200286 & 9799 & 23 & 3065 & 3064.8 & 0.0 & 31.4 & 2285 & 2185 \\ 
  (70,7,14,3,3) & 2673 & 3054 & 3053.7 &  0.0 & 47.8 & 181042 & 11194 & 11 & 3054 & 3054.0 & 0.0 & 28.5 & 2128 & 1242 \\ 
  (80,8,8,1,1) & 18 & 1852 & 1851.9 &  0.0 & 28.6 & 6562 & 3300 & 3 & 1852 & 1852.0 & 0.0 & 19.4 & 240 & 802 \\ 
  (80,8,8,1,2) & 16 & 1896 & 1896.0 &  0.0 & 32.4 & 5670 & 3493 & 3 & 1896 & 1896.0 & 0.0 & 21.1 & 325 & 835 \\ 
  (80,8,8,1,3) & 22 & 1915 & 1915.0 &  0.0 & 30.6 & 7969 & 3345 & 18 & 1915 & 1915.0 & 0.0 & 23.5 & 497 & 1695 \\ 
  (80,8,8,2,1) & 1424 & 3603 & 3602.6 &  0.0 & 29.0 & 56405 & 4443 & 113 & 3603 & 3602.6 & 0.0 & 20.5 & 11148 & 2887 \\ 
  (80,8,8,2,2) & 1078 & 3340 & 3339.7 &  0.0 & 28.4 & 141737 & 3041 & 272 & 3340 & 3339.7 & 0.0 & 20.0 & 39300 & 2192 \\ 
  (80,8,8,2,3) & 1043 & 3312 & 3311.7 &  0.0 & 27.6 & 110767 & 2756 & 73 & 3312 & 3311.7 & 0.0 & 18.9 & 12361 & 2207 \\ 
  (80,8,8,3,1) & 171 & 3860 & 3859.6 &  0.0 & 34.7 & 60099 & 3841 & 1 & 3860 & 3860.0 & 0.0 & 22.6 & 1526 & 556 \\ 
  (80,8,8,3,2) & 635 & 3818 & 3818.0 &  0.0 & 40.6 & 190382 & 6860 & 7 & 3818 & 3818.0 & 0.0 & 25.4 & 7503 & 807 \\ 
  (80,8,8,3,3) & 282 & 4501 & 4500.6 &  0.0 & 40.2 & 121391 & 3197 & 3 & 4501 & 4500.6 & 0.0 & 26.6 & 2666 & 715 \\ 
  (80,8,16,1,1) & TL & 1537 & 1411.0 &  8.2 & 45.1 & 80358 & 64277 & 79 & 1537 & 1536.9 & 0.0 & 36.5 & 2040 & 8442 \\ 
  (80,8,16,1,2) & TL & 1607 & 1464.8 &  8.8 & 44.4 & 84609 & 62889 & 30 & 1586 & 1585.8 & 0.0 & 36.7 & 876 & 5354 \\ 
  (80,8,16,1,3) & TL & 1569 & 1498.4 &  4.5 & 45.7 & 88200 & 52448 & 28 & 1569 & 1569.0 & 0.0 & 36.9 & 869 & 4806 \\ 
  (80,8,16,2,1) & TL & 3547 & 3179.4 & 10.4 & 38.9 & 78427 & 18223 & 357 & 3534 & 3533.7 & 0.0 & 30.5 & 11053 & 11962 \\ 
  (80,8,16,2,2) & TL & 3510 & 2999.2 & 14.6 & 38.0 & 54756 & 14970 & TL & 3494 & 3390.6 & 3.0 & 31.1 & 32400 & 23595 \\ 
  (80,8,16,2,3) & TL & 3243 & 2768.1 & 14.6 & 38.4 & 52184 & 13611 & 1048 & 3196 & 3195.7 & 0.0 & 28.8 & 40677 & 16124 \\ 
  (80,8,16,3,1) & TL & 3685 & 3251.0 & 11.8 & 50.3 & 150519 & 17079 & 15 & 3618 & 3617.7 & 0.0 & 30.9 & 4872 & 1900 \\ 
  (80,8,16,3,2) & TL & 3382 & 3296.0 &  2.5 & 47.5 & 135481 & 14476 & 4 & 3382 & 3382.0 & 0.0 & 28.7 & 1185 & 1045 \\ 
  (80,8,16,3,3) & TL & 3305 & 3017.7 &  8.7 & 53.1 & 141120 & 14903 & 12 & 3296 & 3295.7 & 0.0 & 34.3 & 2826 & 1786 \\ 
\midrule
\end{tabular}
{}
\end{table}

\begin{table} 
\caption{Detailed \WMCIG results under separation option S2: basic and complete setting comparison (continued).}
\label{table:Appendix-MCI-3}
\centering\footnotesize
\begin{tabular}{l|rrrrrrr|rrrrrrr}
\midrule
  & \multicolumn{7}{c|}{B-S2} & \multicolumn{7}{c}{ILDAE-S2} \\ \midrule 
 $(n,B,p,r,ID)$ & t(s.) & UB & LB & Gap(\%) & rGap(\%) & \#Nodes & \#SIC & t(s.)& UB & LB & Gap(\%) & rGap(\%) & \#Nodes & \#SIC\\ \midrule
  (90,9,9,1,1) & 22 & 2343 & 2343.0 &  0.0 & 30.6 & 3950 & 4978 & 7 & 2343 & 2343.0 & 0.0 & 20.4 & 140 & 983 \\ 
  (90,9,9,1,2) & 409 & 1860 & 1859.8 &  0.0 & 29.4 & 49727 & 15633 & 50 & 1860 & 1859.9 & 0.0 & 23.0 & 1442 & 4343 \\ 
  (90,9,9,1,3) & 274 & 2447 & 2446.8 &  0.0 & 29.4 & 29199 & 11829 & 42 & 2447 & 2446.8 & 0.0 & 22.2 & 1020 & 4336 \\ 
  (90,9,9,2,1) & TL & 4579 & 4427.7 &  3.3 & 29.3 & 119802 & 8017 & 604 & 4579 & 4578.6 & 0.0 & 20.7 & 34771 & 4065 \\ 
  (90,9,9,2,2) & TL & 3909 & 3082.0 & 21.2 & 25.7 & 70 & 145 & 1710 & 3887 & 3886.6 & 0.0 & 16.8 & 59925 & 4468 \\ 
  (90,9,9,2,3) & TL & 4273 & 3485.6 & 18.4 & 27.4 & 1110 & 625 & 1019 & 4190 & 4189.6 & 0.0 & 18.4 & 89360 & 5975 \\ 
  (90,9,9,3,1) & TL & 4574 & 4465.5 &  2.4 & 38.7 & 374194 & 9400 & 6 & 4574 & 4573.6 & 0.0 & 24.1 & 6241 & 886 \\ 
  (90,9,9,3,2) & TL & 5147 & 5072.1 &  1.5 & 41.8 & 362315 & 13868 & 4 & 5147 & 5146.6 & 0.0 & 23.6 & 2473 & 831 \\ 
  (90,9,9,3,3) & 3511 & 4137 & 4136.6 &  0.0 & 38.0 & 373095 & 8215 & 10 & 4137 & 4136.7 & 0.0 & 23.7 & 6642 & 1009 \\ 
  (90,9,18,1,1) & TL & 1966 & 1671.6 & 15.0 & 42.3 & 45800 & 62635 & 143 & 1949 & 1948.9 & 0.0 & 34.8 & 2225 & 11965 \\ 
  (90,9,18,1,2) & TL & 1875 & 1661.8 & 11.4 & 42.3 & 53600 & 58722 & 225 & 1867 & 1866.8 & 0.0 & 34.8 & 3675 & 12381 \\ 
  (90,9,18,1,3) & TL & 1709 & 1517.6 & 11.2 & 41.6 & 49606 & 49914 & 93 & 1707 & 1707.0 & 0.0 & 34.9 & 2098 & 8555 \\ 
  (90,9,18,2,1) & TL & 4111 & 3180.5 & 22.6 & 38.7 & 5200 & 3791 & 379 & 3861 & 3860.7 & 0.0 & 31.1 & 7910 & 8876 \\ 
  (90,9,18,2,2) & TL & 4372 & 2989.9 & 31.6 & 39.0 & 900 & 1284 & TL & 4056 & 3941.5 & 2.8 & 28.3 & 43714 & 28317 \\ 
  (90,9,18,2,3) & TL & 3726 & 2980.1 & 20.0 & 36.7 & 6900 & 6812 & 686 & 3630 & 3629.6 & 0.0 & 30.8 & 29127 & 13319 \\ 
  (90,9,18,3,1) & TL & 4389 & 3289.9 & 25.0 & 57.0 & 79511 & 31949 & 35 & 4266 & 4265.6 & 0.0 & 30.5 & 8296 & 3647 \\ 
  (90,9,18,3,2) & TL & 4647 & 3728.5 & 19.8 & 63.3 & 86129 & 23560 & 93 & 4557 & 4556.6 & 0.0 & 34.9 & 18263 & 3319 \\ 
  (90,9,18,3,3) & TL & 4035 & 3255.0 & 19.3 & 57.3 & 68008 & 24866 & 8 & 3969 & 3968.6 & 0.0 & 30.6 & 2738 & 1398 \\ 
  (100,10,10,1,1) & 2071 & 2605 & 2604.7 &  0.0 & 31.1 & 103857 & 28889 & 66 & 2605 & 2604.8 & 0.0 & 22.0 & 2471 & 6374 \\ 
  (100,10,10,1,2) & 1123 & 2486 & 2485.8 &  0.0 & 32.7 & 46843 & 26071 & 32 & 2486 & 2485.9 & 0.0 & 26.8 & 1565 & 3842 \\ 
  (100,10,10,1,3) & 403 & 2424 & 2423.8 &  0.0 & 31.2 & 32035 & 12850 & 15 & 2424 & 2423.9 & 0.0 & 21.8 & 320 & 1655 \\ 
  (100,10,10,2,1) & TL & 4163 & 3551.4 & 14.7 & 27.5 & 3199 & 2043 & 758 & 4091 & 4090.6 & 0.0 & 21 & 28611 & 5282 \\ 
  (100,10,10,2,2) & TL & 4328 & 4148.5 &  4.1 & 25.8 & 67344 & 11120 & 116 & 4317 & 4316.6 & 0.0 & 20.5 & 5268 & 2666 \\ 
  (100,10,10,2,3) & TL & 4271 & 4057.2 &  5.0 & 30.5 & 115011 & 10229 & 324 & 4271 & 4270.6 & 0.0 & 21.6 & 66793 & 5935 \\ 
  (100,10,10,3,1) & TL & 4871 & 4350.3 & 10.7 & 42.4 & 181107 & 24528 & 21 & 4851 & 4850.5 & 0.0 & 25.9 & 12279 & 1562 \\ 
  (100,10,10,3,2) & TL & 5264 & 4831.7 &  8.2 & 43.1 & 137616 & 20799 & 23 & 5264 & 5264.0 & 0.0 & 26.6 & 4581 & 1410 \\ 
  (100,10,10,3,3) & TL & 4430 & 4106.3 &  7.3 & 40.6 & 264138 & 14489 & 24 & 4430 & 4429.6 & 0.0 & 26.7 & 15221 & 1415 \\ 
  (100,10,20,1,1) & TL & 2274 & 1812.1 & 20.3 & 43.9 & 33100 & 62332 & TL & 2243 & 2187.4 & 2.5 & 36.0 & 9237 & 49097 \\ 
  (100,10,20,1,2) & TL & 2575 & 2080.0 & 19.2 & 41.3 & 32700 & 51822 & 1258 & 2538 & 2537.8 & 0.0 & 35.5 & 5245 & 27756 \\ 
  (100,10,20,1,3) & TL & 2101 & 1659.7 & 21.0 & 42.0 & 30204 & 55337 & 2701 & 2074 & 2073.8 & 0.0 & 35.7 & 9026 & 37033 \\ 
  (100,10,20,2,1) & TL & 4378 & 3735.5 & 14.7 & 42.2 & 55400 & 24859 & 311 & 4278 & 4277.6 & 0.0 & 31.9 & 9821 & 5672 \\ 
  (100,10,20,2,2) & TL & 4523 & 3606.2 & 20.3 & 39.8 & 2881 & 3736 & 215 & 4279 & 4278.6 & 0.0 & 31.1 & 6095 & 4390 \\ 
  (100,10,20,2,3) & TL & 4681 & 3897.3 & 16.7 & 42.8 & 34217 & 14905 & 972 & 4521 & 4520.6 & 0.0 & 31.4 & 26589 & 8652 \\ 
  (100,10,20,3,1) & TL & 5022 & 3364.9 & 33.0 & 67.6 & 74907 & 28993 & 349 & 4911 & 4910.5 & 0.0 & 35.2 & 38637 & 5308 \\ 
  (100,10,20,3,2) & TL & 4917 & 3239.4 & 34.1 & 63.1 & 68400 & 30693 & 63 & 4523 & 4523.0 & 0.0 & 34.9 & 16533 & 3055 \\ 
  (100,10,20,3,3) & TL & 4825 & 3120.3 & 35.3 & 64.5 & 66800 & 26868 & 989 & 4609 & 4608.6 & 0.0 & 36.0 & 47811 & 8275 \\  \midrule
 Average & 1606.5 & 2778.4 & 2569.9 & 5.4 & 37.5 & 62216.6 & 12235.8 & 290.3 & 2753.7 & 2750.9 & 0.1 & 25.9 & 9114.8 & 4503.6 \\ 
   \midrule
\end{tabular}
{}
\end{table}

\begin{table}
\caption{Detailed \BIIG results under separation option S2: basic and complete setting comparison.}
\label{table:Appendix-BII-1}
{\centering\small
\begin{tabular}{l|rrrrrrr|rrrrrrr}
\midrule
  & \multicolumn{7}{c|}{B-S2} & \multicolumn{7}{c}{ILDAE-S2} \\ \midrule 
 $(n,m,B,p,d,ID)$ & t(s.) & UB & LB & Gap(\%) & rGap(\%) & \#Nodes & \#SIC & t(s.)& UB & LB & Gap(\%) & rGap(\%) & \#Nodes & \#SIC\\ \midrule
(20,40,5,5,0.07,1) & 0 & 7 &   7.0 &  0.0 & 43.0 & 156 & 173 & 0 & 7 &   7.0 &  0.0 & 28.4 & 92 & 305 \\ 
  (20,40,5,5,0.07,2) & 1 & 9 &   9.0 &  0.0 & 43.9 & 110 & 146 & 0 & 9 &   9.0 &  0.0 & 30.4 & 61 & 226 \\ 
  (20,40,5,5,0.10,1) & 0 & 9.8 &   9.8 &  0.0 & 43.5 & 121 & 120 & 0 & 9.8 &   9.8 &  0.0 & 22.8 & 79 & 207 \\ 
  (20,40,5,5,0.10,2) & 0 & 8.8 &   8.8 &  0.0 & 41.7 & 35 & 53 & 0 & 8.8 &   8.8 &  0.0 & 23.0 & 17 & 62 \\ 
  (20,40,5,5,0.15,1) & 0 & 12.1 &  12.1 &  0.0 & 33.1 & 48 & 57 & 0 & 12.1 &  12.1 &  0.0 & 22.1 & 31 & 104 \\ 
  (20,40,5,5,0.15,2) & 1 & 19.5 &  19.5 &  0.0 & 46.0 & 255 & 220 & 1 & 19.5 &  19.5 &  0.0 & 31.2 & 125 & 386 \\ 
  (20,40,10,5,0.07,1) & 0 & 11 &  11.0 &  0.0 & 31.3 & 48 & 101 & 0 & 11 &  11.0 &  0.0 & 22.8 & 37 & 125 \\ 
  (20,40,10,5,0.07,2) & 0 & 8.7 &   8.7 &  0.0 & 26.5 & 40 & 81 & 0 & 8.7 &   8.7 &  0.0 & 21.0 & 26 & 111 \\ 
  (20,40,10,5,0.10,1) & 0 & 12.6 &  12.6 &  0.0 & 32.5 & 57 & 92 & 0 & 12.6 &  12.6 &  0.0 & 24.2 & 27 & 78 \\ 
  (20,40,10,5,0.10,2) & 0 & 16.1 &  16.1 &  0.0 & 38.8 & 107 & 184 & 0 & 16.1 &  16.1 &  0.0 & 27.0 & 78 & 216 \\ 
  (20,40,10,5,0.15,1) & 0 & 15.6 &  15.6 &  0.0 & 32.6 & 78 & 147 & 0 & 15.6 &  15.6 &  0.0 & 21.8 & 46 & 139 \\ 
  (20,40,10,5,0.15,2) & 1 & 21.1 &  21.1 &  0.0 & 41.3 & 335 & 694 & 0 & 21.1 &  21.1 &  0.0 & 29.0 & 60 & 265 \\ 
  (20,100,5,5,0.07,1) & 0 & 18 &  18.0 &  0.0 & 40.9 & 106 & 147 & 0 & 18 &  18.0 &  0.0 & 29.8 & 57 & 212 \\ 
  (20,100,5,5,0.07,2) & 1 & 17.1 &  17.1 &  0.0 & 33.6 & 53 & 89 & 0 & 17.1 &  17.1 &  0.0 & 18.1 & 28 & 108 \\ 
  (20,100,5,5,0.10,1) & 1 & 29.1 &  29.1 &  0.0 & 41.2 & 92 & 127 & 4 & 29.1 &  29.1 &  0.0 & 25.8 & 54 & 213 \\ 
  (20,100,5,5,0.10,2) & 1 & 30.7 &  30.7 &  0.0 & 48.1 & 198 & 245 & 2 & 30.7 &  30.7 &  0.0 & 30.4 & 94 & 387 \\ 
  (20,100,5,5,0.15,1) & 0 & 43.2 &  43.2 &  0.0 & 40.4 & 264 & 183 & 0 & 43.2 &  43.2 &  0.0 & 21.3 & 52 & 319 \\ 
  (20,100,5,5,0.15,2) & 0 & 38.8 &  38.8 &  0.0 & 46.6 & 115 & 156 & 0 & 38.8 &  38.8 &  0.0 & 29.1 & 40 & 173 \\ 
  (20,100,10,5,0.07,1) & 5 & 22.2 &  22.2 &  0.0 & 29.1 & 27 & 65 & 0 & 22.2 &  22.2 &  0.0 & 22.5 & 22 & 72 \\ 
  (20,100,10,5,0.07,2) & 2 & 25.1 &  25.1 &  0.0 & 34.4 & 191 & 339 & 15 & 25.1 &  25.1 &  0.0 & 29.2 & 106 & 394 \\ 
  (20,100,10,5,0.10,1) & 6 & 45.6 &  45.6 &  0.0 & 42.9 & 1187 & 1622 & 3 & 45.6 &  45.6 &  0.0 & 29.0 & 397 & 1345 \\ 
  (20,100,10,5,0.10,2) & 7 & 42.3 &  42.3 &  0.0 & 39.9 & 292 & 529 & 0 & 42.3 &  42.3 &  0.0 & 27.2 & 129 & 400 \\ 
  (20,100,10,5,0.15,1) & 0 & 54.1 &  54.1 &  0.0 & 45.1 & 208 & 393 & 0 & 54.1 &  54.1 &  0.0 & 32.7 & 78 & 224 \\ 
  (20,100,10,5,0.15,2) & 3 & 49.6 &  49.6 &  0.0 & 42.2 & 392 & 763 & 0 & 49.6 &  49.6 &  0.0 & 29.6 & 116 & 377 \\ 
  (20,200,5,5,0.07,1) & 3 & 34.1 &  34.1 &  0.0 & 37.6 & 104 & 134 & 0 & 34.1 &  34.1 &  0.0 & 24.7 & 69 & 206 \\ 
  (20,200,5,5,0.07,2) & 0 & 34.4 &  34.4 &  0.0 & 41.9 & 54 & 80 & 0 & 34.4 &  34.4 &  0.0 & 23.6 & 23 & 92 \\ 
  (20,200,5,5,0.10,1) & 1 & 50.1 &  50.1 &  0.0 & 35.8 & 114 & 130 & 1 & 50.1 &  50.1 &  0.0 & 25.2 & 30 & 203 \\ 
  (20,200,5,5,0.10,2) & 0 & 63.8 &  63.8 &  0.0 & 48.0 & 202 & 241 & 3 & 63.8 &  63.8 &  0.0 & 31.9 & 156 & 507 \\ 
  (20,200,5,5,0.15,1) & 4 & 68.2 &  68.2 &  0.0 & 40.5 & 128 & 162 & 0 & 68.2 &  68.2 &  0.0 & 26.3 & 31 & 171 \\ 
  (20,200,5,5,0.15,2) & 0 & 90.7 &  90.7 &  0.0 & 47.4 & 173 & 203 & 0 & 90.7 &  90.7 &  0.0 & 27.2 & 49 & 311 \\ 
  (20,200,10,5,0.07,1) & 1 & 40.3 &  40.3 &  0.0 & 29.7 & 103 & 201 & 0 & 40.3 &  40.3 &  0.0 & 23.6 & 60 & 235 \\ 
  (20,200,10,5,0.07,2) & 0 & 48.4 &  48.4 &  0.0 & 32.8 & 121 & 219 & 0 & 48.4 &  48.4 &  0.0 & 29.2 & 82 & 278 \\ 
  (20,200,10,5,0.10,1) & 3 & 80.2 &  80.2 &  0.0 & 45.3 & 978 & 1462 & 5 & 80.2 &  80.2 &  0.0 & 33.0 & 264 & 912 \\ 
  (20,200,10,5,0.10,2) & 4 & 89.9 &  89.9 &  0.0 & 43.2 & 736 & 1215 & 1 & 89.9 &  89.9 &  0.0 & 30.5 & 201 & 771 \\ 
  (20,200,10,5,0.15,1) & 8 & 106.7 & 106.7 &  0.0 & 40.5 & 1025 & 1670 & 2 & 106.7 & 106.7 &  0.0 & 26.6 & 182 & 820 \\ 
  (20,200,10,5,0.15,2) & 1 & 101.9 & 101.9 &  0.0 & 44.0 & 369 & 606 & 1 & 101.9 & 101.9 &  0.0 & 27.4 & 134 & 401 \\ 
     \midrule
\end{tabular}}
{}
\end{table}
  
  \begin{table}
\caption{Detailed \BIIG results under separation option S2: basic and complete setting comparison (continued).}
\label{table:Appendix-BII-2}
\centering\footnotesize
\begin{tabular}{l|rrrrrrr|rrrrrrr}
\midrule
  & \multicolumn{7}{c|}{B-S2} & \multicolumn{7}{c}{ILDAE-S2} \\ \midrule 
 $(n,m,B,p,d,ID)$ & t(s.) & UB & LB & Gap(\%) & rGap(\%) & \#Nodes & \#SIC & t(s.)& UB & LB & Gap(\%) & rGap(\%) & \#Nodes & \#SIC\\ \midrule
  (50,100,10,10,0.07,1) & 129 & 34.9 &  34.9 &  0.0 & 44.8 & 5011 & 10904 & 17 & 34.9 &  34.9 &  0.0 & 29.9 & 518 & 2742 \\ 
  (50,100,10,10,0.07,2) & 1399 & 39.7 &  39.7 &  0.0 & 50.4 & 30357 & 39996 & 100 & 39.7 &  39.7 &  0.0 & 34.7 & 1711 & 10919 \\ 
  (50,100,10,10,0.10,1) & 544 & 49.8 &  49.8 &  0.0 & 53.7 & 15896 & 27497 & 43 & 49.8 &  49.8 &  0.0 & 32.1 & 989 & 5323 \\ 
  (50,100,10,10,0.10,2) & 77 & 42.7 &  42.7 &  0.0 & 46.4 & 3854 & 9208 & 11 & 42.7 &  42.7 &  0.0 & 34.6 & 380 & 2187 \\ 
  (50,100,10,10,0.15,1) & 2215 & 66.6 &  66.6 &  0.0 & 59.6 & 52408 & 52320 & 221 & 66.6 &  66.6 &  0.0 & 38.4 & 2481 & 11722 \\ 
  (50,100,10,10,0.15,2) & 2705 & 64 &  64.0 &  0.0 & 51.7 & 72302 & 53658 & 182 & 64 &  64.0 &  0.0 & 36.2 & 2145 & 12375 \\ 
  (50,100,20,10,0.07,1) & TL & 56.9 &  49.8 & 12.4 & 42.4 & 10400 & 50273 & 119 & 55.5 &  55.5 &  0.0 & 27.5 & 2421 & 12942 \\ 
  (50,100,20,10,0.07,2) & TL & 44.7 &  41.2 &  7.9 & 41.1 & 14311 & 60255 & 144 & 44.4 &  44.4 &  0.0 & 25.8 & 3021 & 14765 \\ 
  (50,100,20,10,0.10,1) & TL & 75.7 &  55.6 & 26.5 & 54.0 & 9680 & 51406 & TL & 74.5 &  70.5 &  5.3 & 27.1 & 12537 & 64302 \\ 
  (50,100,20,10,0.10,2) & TL & 75.3 &  57.5 & 23.7 & 50.6 & 10018 & 49087 & TL & 73.9 &  72.0 &  2.5 & 25.6 & 16016 & 59568 \\ 
  (50,100,20,10,0.15,1) & TL & 86.4 &  54.9 & 36.4 & 60.8 & 5604 & 25833 & TL & 85 &  76.9 &  9.5 & 28.8 & 7186 & 28639 \\ 
  (50,100,20,10,0.15,2) & TL & 83.3 &  56.1 & 32.7 & 60.2 & 10201 & 53265 & TL & 81.4 &  78.1 &  4.1 & 29.9 & 15758 & 69968 \\ 
  (50,250,10,10,0.07,1) & 433 & 88.6 &  88.6 &  0.0 & 48.8 & 14045 & 27009 & 50 & 88.6 &  88.6 &  0.0 & 33.9 & 1262 & 8492 \\ 
  (50,250,10,10,0.07,2) & 65 & 83.9 &  83.9 &  0.0 & 44.9 & 2358 & 6228 & 22 & 83.9 &  83.9 &  0.0 & 28.4 & 148 & 1151 \\ 
  (50,250,10,10,0.10,1) & 989 & 119.6 & 119.6 &  0.0 & 54.5 & 15190 & 33535 & 56 & 119.6 & 119.6 &  0.0 & 36.7 & 778 & 5223 \\ 
  (50,250,10,10,0.10,2) & TL & 139.6 &  80.2 & 42.5 & 52.6 & 280 & 1210 & 375 & 136.3 & 136.3 &  0.0 & 37.2 & 2619 & 16895 \\ 
  (50,250,10,10,0.15,1) & TL & 161.9 & 144.6 & 10.7 & 61.0 & 57500 & 63497 & 241 & 161.9 & 161.9 &  0.0 & 39.3 & 3169 & 18397 \\ 
  (50,250,10,10,0.15,2) & TL & 168.2 & 143.6 & 14.6 & 58.6 & 59207 & 66619 & 816 & 167.7 & 167.7 &  0.0 & 40.3 & 6174 & 28222 \\ 
  (50,250,20,10,0.07,1) & TL & 131.1 & 113.3 & 13.5 & 45.8 & 9100 & 49674 & 316 & 129.2 & 129.2 &  0.0 & 30.2 & 3877 & 21942 \\ 
  (50,250,20,10,0.07,2) & TL & 140.5 & 112.8 & 19.7 & 46.0 & 8030 & 51621 & TL & 138.4 & 131.5 &  5.0 & 26.5 & 8600 & 57108 \\ 
  (50,250,20,10,0.10,1) & TL & 178.8 & 126.5 & 29.3 & 53.9 & 7146 & 46735 & TL & 176.1 & 160.8 &  8.7 & 29.4 & 6040 & 50014 \\ 
  (50,250,20,10,0.10,2) & TL & 173.4 & 119.6 & 31.0 & 57.6 & 7101 & 48288 & TL & 172.4 & 154.9 & 10.1 & 31.3 & 6871 & 53266 \\ 
  (50,250,20,10,0.15,1) & TL & 207.3 & 120.5 & 41.9 & 56.1 & 1901 & 11403 & TL & 196.4 & 191.4 &  2.6 & 28.1 & 13240 & 57502 \\ 
  (50,250,20,10,0.15,2) & TL & 206.4 & 125.7 & 39.1 & 61.8 & 6205 & 35725 & TL & 201.3 & 185.1 &  8.1 & 30.5 & 9600 & 53041 \\ 
  (50,500,10,10,0.07,1) & TL & 218.4 &  93.2 & 57.3 & 57.3 & 0 & 6 & 630 & 204.9 & 204.9 &  0.0 & 37.7 & 3421 & 23860 \\ 
  (50,500,10,10,0.07,2) & 616 & 182.1 & 182.1 &  0.0 & 48.3 & 10642 & 24438 & 34 & 182.1 & 182.1 &  0.0 & 35.3 & 721 & 5251 \\ 
  (50,500,10,10,0.10,1) & 1210 & 218.4 & 218.4 &  0.0 & 53.8 & 19088 & 40602 & 32 & 218.4 & 218.4 &  0.0 & 37.5 & 679 & 4789 \\ 
  (50,500,10,10,0.10,2) & TL & 281 & 165.0 & 41.3 & 54.3 & 660 & 3045 & TL & 278.5 & 268.4 &  3.6 & 35.8 & 9910 & 67973 \\ 
  (50,500,10,10,0.15,1) & TL & 306.2 & 284.6 &  7.1 & 55.1 & 45200 & 63211 & 171 & 306.2 & 306.2 &  0.0 & 35.8 & 1696 & 10761 \\ 
  (50,500,10,10,0.15,2) & TL & 316.9 & 275.2 & 13.2 & 59.4 & 43700 & 62129 & 178 & 314.6 & 314.6 &  0.0 & 41.0 & 2430 & 14972 \\ 
  (50,500,20,10,0.07,1) & TL & 278 & 229.6 & 17.4 & 44.7 & 9202 & 58121 & 848 & 272.3 & 272.3 &  0.0 & 25.8 & 4694 & 28875 \\ 
  (50,500,20,10,0.07,2) & TL & 280.3 & 217.0 & 22.6 & 46.8 & 6830 & 48041 & TL & 269.3 & 258.9 &  3.9 & 28.9 & 6216 & 52577 \\ 
  (50,500,20,10,0.10,1) & TL & 374.4 & 248.0 & 33.7 & 54.8 & 7281 & 44691 & TL & 366.6 & 332.4 &  9.3 & 30.1 & 6006 & 48437 \\ 
  (50,500,20,10,0.10,2) & TL & 353.5 & 238.0 & 32.7 & 53.7 & 8627 & 52592 & TL & 351.2 & 321.6 &  8.4 & 31.5 & 5812 & 53780 \\ 
  (50,500,20,10,0.15,1) & TL & 443.1 & 153.5 & 65.3 & 65.3 & 0 & 30 & TL & 436 & 345.4 & 20.8 & 28.4 & 400 & 2695 \\ 
  (50,500,20,10,0.15,2) & TL & 432.8 & 233.2 & 46.1 & 63.4 & 4800 & 27900 & TL & 427.2 & 384.3 & 10.0 & 29.6 & 8938 & 62484 \\ 
     \midrule
\end{tabular}
{}
\end{table}

\begin{table}
\caption{Detailed \BIIG results under separation option S2: basic and complete setting comparison (continued).}
\label{table:Appendix-BII-3}
\centering\footnotesize
\begin{tabular}{l|rrrrrrr|rrrrrrr}
\midrule
  & \multicolumn{7}{c|}{B-S2} & \multicolumn{7}{c}{ILDAE-S2} \\ \midrule 
 $(n,m,B,p,d,ID)$ & t(s.) & UB & LB & Gap(\%) & rGap(\%) & \#Nodes & \#SIC & t(s.)& UB & LB & Gap(\%) & rGap(\%) & \#Nodes & \#SIC\\ \midrule  
  (100,200,10,10,0.07,1) & TL & 88.8 &  84.8 &  4.5 & 41.8 & 104323 & 49295 & 728 & 88.2 &  88.2 &  0.0 & 30.8 & 4440 & 25039 \\ 
  (100,200,10,10,0.07,2) & TL & 95.6 &  93.1 &  2.6 & 39.7 & 113777 & 43135 & 1016 & 95.6 &  95.6 &  0.0 & 30.5 & 5534 & 23318 \\ 
  (100,200,10,10,0.10,1) & TL & 127.3 &  76.3 & 40.1 & 44.9 & 150 & 505 & 2269 & 125.6 & 125.5 &  0.0 & 31.7 & 9037 & 30262 \\ 
  (100,200,10,10,0.10,2) & TL & 128.2 & 117.3 &  8.5 & 47.3 & 72200 & 43895 & 2291 & 128.2 & 128.2 &  0.0 & 33.4 & 9395 & 30865 \\ 
  (100,200,10,10,0.15,1) & TL & 148.8 & 121.5 & 18.3 & 52.1 & 49700 & 25067 & 2694 & 148.2 & 148.2 &  0.0 & 35.1 & 17436 & 38490 \\ 
  (100,200,10,10,0.15,2) & TL & 145.3 & 114.9 & 20.9 & 49.8 & 22400 & 15742 & 1039 & 144.1 & 144.1 &  0.0 & 33.1 & 6447 & 22374 \\ 
  (100,500,10,10,0.07,1) & TL & 233.1 & 217.6 &  6.6 & 44.0 & 80762 & 48886 & 1811 & 233.1 & 233.1 &  0.0 & 29.4 & 6563 & 39853 \\ 
  (100,500,10,10,0.07,2) & TL & 229 & 215.5 &  5.9 & 48.1 & 88100 & 57686 & 975 & 228.3 & 228.3 &  0.0 & 31.7 & 5120 & 30972 \\ 
  (100,500,10,10,0.10,1) & TL & 315 & 194.5 & 38.3 & 49.4 & 1110 & 3239 & TL & 313.8 & 299.1 &  4.7 & 36.4 & 7886 & 53222 \\ 
  (100,500,10,10,0.10,2) & TL & 311.7 & 215.4 & 30.9 & 49.0 & 3150 & 3779 & TL & 310.5 & 294.7 &  5.1 & 31.3 & 7638 & 47035 \\ 
  (100,500,10,10,0.15,1) & TL & 363.8 & 213.4 & 41.3 & 51.7 & 800 & 1771 & 3487 & 361 & 361.0 &  0.0 & 34.6 & 14020 & 47130 \\ 
  (100,500,10,10,0.15,2) & TL & 394 & 218.8 & 44.5 & 52.0 & 286 & 730 & TL & 392.3 & 299.2 & 23.7 & 31.7 & 251 & 1266 \\ 
  (100,1000,10,10,0.07,1) & TL & 473.4 & 418.2 & 11.7 & 46.0 & 85802 & 45550 & TL & 475.7 & 449.9 &  5.4 & 32.3 & 3869 & 40097 \\ 
  (100,1000,10,10,0.07,2) & TL & 471.4 & 417.1 & 11.5 & 45.5 & 85407 & 50176 & TL & 472.3 & 448.0 &  5.1 & 32.1 & 4155 & 46077 \\ 
  (100,1000,10,10,0.10,1) & TL & 575 & 395.5 & 31.2 & 47.5 & 4400 & 9462 & TL & 571.7 & 536.4 &  6.2 & 34.5 & 5137 & 42547 \\ 
  (100,1000,10,10,0.10,2) & TL & 532.4 & 431.6 & 18.9 & 51.3 & 28800 & 28567 & 2952 & 532.4 & 532.3 &  0.0 & 34.6 & 7006 & 51664 \\ 
  (100,1000,10,10,0.15,1) & TL & 758.2 & 410.5 & 45.9 & 49.6 & 80 & 383 & TL & 745.3 & 669.9 & 10.1 & 36.2 & 3990 & 24286 \\ 
  (100,1000,10,10,0.15,2) & TL & 753.7 & 389.4 & 48.3 & 48.3 & 0 & 45 & TL & 735.7 & 673.7 &  8.4 & 36.4 & 4782 & 29665 \\ \midrule
  Average & 1836.0 & 152.5 & 113.9 & 12.8 & 46.6 & 14822.3 & 19900.2 & 1185.7 & 151.0 & 143.9 & 2.0 & 30.1 & 3381.1 & 18652.0 \\ 
   \midrule
\end{tabular}
{}
\end{table}

\end{landscape}

\end{document}